\documentclass[a4paper]{amsart}
\usepackage[colorlinks,linkcolor=black,anchorcolor=black,citecolor=black,hyperindex=true,CJKbookmarks=true]{hyperref}
\usepackage{amssymb,amsmath,amsthm,amstext,amsfonts}
\usepackage[dvips]{graphicx}
\usepackage{psfrag}
\usepackage{url}
\usepackage{amsfonts}
\usepackage{amsmath}
\usepackage{mathrsfs}
\usepackage{graphicx}
\usepackage{amsmath,amsthm,amssymb,amscd}
\usepackage{color}
\usepackage{enumerate}

\usepackage{epsfig}

\makeatletter \@addtoreset{equation}{section} \makeatother

\renewcommand\thetable{\thesection.\@arabic\c@table}

\theoremstyle{plain}

\newtheorem{theorem}{Theorem}[section]
\newtheorem{proposition}{Proposition}[section]
\newtheorem{lemma}{Lemma}[section]
\newtheorem{corollary}{Corollary}[section]
\newtheorem{definition}{Definition}[section]
\newtheorem{remark}{Remark}[section]

\newtheorem{Thm}{Theorem}[section]

\newtheorem{Lem}[Thm]{Lemma}
\newtheorem{Prop}[Thm]{Proposition}

\theoremstyle{remark}
\newtheorem{Def}[Thm] {Definition}
\newtheorem{Rem}[Thm] {Remark}

\long\def\begcom#1\endcom{}

\newcommand{\length}{\operatorname{\length}}

\def\length{\operatorname{length}}

\def\top{\operatorname{top}}

\newcommand{\bl} {\begin{lemma}}
\newcommand{\el} {\end{lemma}}
\newcommand{\bt} {\begin{theorem}}
\newcommand{\et} {\end{theorem}}

\newcommand{\bp}{\begin{proof}}
\newcommand{\ep}{\end{proof}}
\newcommand  {\ee} {\end{equation}}
\newcommand  {\beq} {\begin{eqnarray*}}
\newcommand  {\eeq} {\end{eqnarray*}}

\newcommand  {\bd} {\begin{definition}}
\newcommand  {\ed} {\end{definition}}

\newcommand{\cM}{\mathcal{M}}
\newcommand{\cB}{\mathcal{B}}

\def\ep{\noindent{\hfill $\Box$}}

\begin{document}

\title[Topological structures on saturated sets, optimal orbits and equilibrium states]{Topological structures on saturated sets, optimal orbits and equilibrium states}

\author{Xiaobo Hou}
\address{School of Mathematical Sciences,  Fudan University\\Shanghai 200433, People's Republic of China}
\email{20110180003@fudan.edu.cn}

\author{Xueting Tian}
\address{School of Mathematical Sciences,  Fudan University\\Shanghai 200433, People's Republic of China}
\email{xuetingtian@fudan.edu.cn}

\author{Yiwei Zhang}
\address{School of Mathematics and Statistics, Center for
Mathematical Sciences, Hubei Key Laboratory of Engineering Modeling
and Scientific Computing, Huazhong University of Sciences and
Technology, Wuhan 430074, China}
\email{yiweizhang@hust.edu.cn}

\begin{abstract}


Pfister and Sullivan proved that if a topological dynamical system $(X,T)$ satisfies almost product property and uniform separation property, then for each nonempty compact 
subset $K$ of invariant measures, the entropy of saturated set $G_{K}$ satisfies 
\begin{equation}\label{Bowen's topological entropy}
	h_{top}^{B}(T,G_{K})=\inf\{h(T,\mu):\mu\in K\},
\end{equation}
where $h_{top}^{B}(T,G_{K})$ is Bowen's topological entropy of $T$ on $G_{K}$, and $h(T,\mu)$ is the Kolmogorov-Sinai entropy of $\mu$. In this paper, we investigate topological complexity of $G_{K}$ by replacing Bowen's topological entropy with upper capacity entropy and packing entropy in (\ref{Bowen's topological entropy}) and obtain the following formulas:
\begin{equation*}
	h_{top}^{UC}(T,G_{K})=h_{top}(T,X)\ \mathrm{and}\ h_{top}^{P}(T,G_{K})=\sup\{h(T,\mu):\mu\in K\},
\end{equation*}
where $h_{top}^{UC}(T,G_{K})$ is the upper capacity entropy of $T$ on $G_{K}$ and $h_{top}^{P}(T,G_{K})$ is the packing entropy of $T$ on $G_{K}.$ In the proof of these two formulas, uniform separation property is unnecessary. 

As applications, when $(X,T)$ is a transitive Anosov diffeomorphism on a compact manifold, we show that (1) for any continuous function $f,$ $h_{top}^{UC}(T,S^{op}_{f})=h_{top}(T,X)$ where $S^{op}_{f}$ is the set of initial values of $f$-optimal orbits; (2) there exists a Baire generic subset $\mathcal{F}$ in the space of continuous functions and an open and dense subset $\mathcal{G}$ in the space of H\"{o}lder continuous or $C^{1}$ smooth functions such that for any $f\in \mathcal{F}$ or $f\in \mathcal{G}$, $h_{top}^{B}(T,S^{op}_{f})=h_{top}^{P}(T,S^{op}_{f})=0;$ (3) for any $f\in \mathcal{G},$ $h_{top}^{UC}(T,S_{f}^{MR})=0$ where $S_{f}^{MR}$ is the set of initial values of measure-recurrent optimal orbits; (4) there exists a dense subset in the space of continuous functions such that for any $f$ in the subset, $h_{top}^{B}(T,S^{op}_{f})=h_{top}^{P}(T,S^{op}_{f})>0$. For equilibrium states, we prove that (1) for any continuous function $f$ and equilibrium state $\mu$ of $f,$ $h_{top}^{UC}(T,G_{\mu})=h_{top}(T,X)$; (2) if $f$ is H\"{o}lder continuous and is not cohomologous to a constant, then $h_{top}^{B}(T,G_{\mu})=h_{top}^{P}(T,G_{\mu})<h_{top}(T,X).$ We also prove that for any $f\in \mathcal{G},$ there exists an uncountable Li-Yorke scrambled set in $S^{op}_{f}.$ 

\end{abstract}

\maketitle
\section{Introduction}
\subsection{Motivations and main theorems}

Throughout this paper, a topological dynamical system (abbr. TDS) $(X,T)$ means always that $(X,d)$ is a compact metric space, and $T:X\to X$ is a continuous map. Let $\cM(X)$, $\cM(X,T)$, $\cM^{e}(X,T)$ denote the spaces of probability measures, $T$-invariant, $T$-ergodic probability measures, respectively. Let $\mathbb{N},$ $\mathbb{N^{+}},$ $\mathbb{R}$ and $\mathbb{R}^{+}$ denote non-negative integers, positive integers, real numbers and positive real numbers, respectively.

In the theory of multifractal analysis, one can partition the space $X$ into different fractal sets according to different asymptotic behavior of orbits $\text{orb}(x,T)=\left\{T^{n}(x):n \in \mathbb{N}\right\}$, there are fruitful researches on characterizing the complexity of these fractal sets. For examples, there are many fractal sets associated with Birkhoff averages. Given a point $x\in X$, let $M_{x}$ be the set of accumulation points of the empirical measure of the Birkhoff average $\frac{1}{n}\sum_{i=0}^{n-1}\delta_{T^{i}x}$, where $\delta_{y}$ is the Dirac measure at point $y$. A point $x \in X$ is said to be \emph{generic} for some invariant measure $\mu$ if $M_{x}=\{\mu\}$ (or equivalently, Birkhoff averages of all continuous functions converge to the integral of $\mu).$ Let $G_{\mu}$ denote the set of all generic points for $\mu.$ In \cite{Bowen} Bowen proved a remarkable result that 
\begin{equation}\label{equ:bowen's result}
	h_{top}^{B}(T,G_{\mu})=h(T,\mu),~~~\mbox{if}~~\mu\in\cM^{e}(X,T),
\end{equation}
where $h_{top}^{B}(T,G_{\mu})$ is the topological entropy (in the sense of Bowen \cite{Bowen}) of $T$ on $G_{\mu}$, and $h(T,\mu)$ is the Kolmogorov-Sinai entropy of $\mu$. 
For each nonempty compact connected subset $K\subset\cM(X,T)$, denote by $G_{K}:=\{x\in X:M_{x}=K\}$,
the \emph{saturated set of $K$}. Note that for any $x\in X$, $M_{x}$ is always a nonempty compact connected subset of $\cM(X,T)$ \cite[Proposition 3.8]{DGS}, so $G_{K}\neq\emptyset$ requires that $K$ is a nonempty compact connected set. The systematic researches on quantifying the size of saturated sets for chaotic systems 
were developed extensively by Sigmund \cite{SigSpe}, Pfister and Sullivan \cite{PS}. The nonempty of $G_{K}$ is proved by Sigmund \cite{SigSpe} for systems with specification property. Pfister and Sullivan \cite{PS} proved that if $T$ satisfies \emph{almost product property} and \emph{uniform separation property}, $K\subseteq \mathcal{M}(X,T)$ is a nonempty compact connected set, then
\begin{equation}\label{equ:PS's result}
  h_{top}^{B}(T,G_{K})=\inf\{h(T,\mu):\mu\in K\}.
\end{equation}
Recently, in \cite[Theorem 1.4]{HTW} the authors considered \emph{transitively-saturated set} $G^{T}_{K}:=G_{K}\cap Tran$, where $Tran$ is the set of transitive points, and it was proved that if further there is an invariant measure with full support, then
\begin{equation}\label{equ:Tian's result}
G^{T}_{K}\neq\emptyset,~~~~h_{top}^{B}(T,G^{T}_{K})=\inf\{h(T,\mu):\mu\in K\}.
\end{equation}

Apart from Bowen's topological entropy $h_{top}^{B}$, there are many concepts for characterizing the size of non-compact subset $Z\subset X$. For examples, the \emph{upper capacity entropy} $h_{top}^{UC}$ is a straightforward generalization of the Adler-Konheim-McAndrew definition of the classical topological entropy \cite{AKM}, and the \emph{packing entropy} $h_{top}^{P}$ was introduced by Feng and Huang \cite{Feng-Huang} in a way resembling packing dimension in fractal geometry theory. Moreover, one always has $
h_{top}^{B}(Z)\leq h_{top}^{P}(Z)\leq h_{top}^{UC}(Z),\forall Z\subset X$. So, a natural question arises, {\bf whether $h_{top}^{B}$ can be replaced by $h_{top}^{P}$ and $h_{top}^{UC}$ in Equation \eqref{equ:PS's result} and \eqref{equ:Tian's result}?}

In the present paper, we give affirmative answers to this kind of questions. Before we state our main result, we note that if a subset $Z\subseteq X$ is dense in $X$, then $h_{top}^{UC}(T,Z)=h_{top}(T,X)$ (see Proposition \ref{prop-AA}). In particular, since $G_{K}^{T}$ is $T$-invariant and $G_{K}^{T}\subset Tran$, then $G_{K}^{T}$ is either empty or dense in $X$. So $G_{K}^{T}$ is either empty or $h_{top}^{UC}(T,G_{K})=h_{top}^{UC}(T,G_{K}^{T})=h_{top}(T,X)$. We concentrate on describing the local behavior of the saturated sets in the sequel. To state our main theorems, define the \emph{omega limit set} of $x$ by $\omega_T(x):=\bigcap_{n=0}^{\infty}\overline{\bigcup_{k=n}^{\infty}\{T^{k}x\}}$, and the \emph{measure center} by $C_T(X):=\overline{\cup_{\mu\in\mathcal{M}(X,T)}S_{\mu}},$ where $S_\mu:=\{x\in X:\mu(U)>0\ \text{for any neighborhood}\ U\ \text{of}\ x\}$ is the \emph{support} of $\mu.$ 
\begin{theorem}\label{entropy}
Suppose that $(X,T)$ satisfies the almost product property. If $K\subseteq \mathcal{M}(X,T)$ is a nonempty compact connected set, then 
\begin{itemize}
	\item for any non-empty open set $U\subseteq X$ with $U\cap C_T(X)\neq\emptyset,$ we have
	\begin{equation*}
		\begin{split}
			&h_{top}^{UC}(T,G_{K}\cap U\cap \{x\in X:C_T(X)\subset \omega_T(x)\})\\
			=&h_{top}^{UC}(T,G_{K}\cap \{x\in X:C_T(X)\subset \omega_T(x)\})\\
			=&h_{top}^{UC}(T,G_{K}\cap U)=h_{top}^{UC}(T,G_{K})=h_{top}(T,X);
		\end{split}
	\end{equation*}
	\item in particular, if further there is an invariant measure with full support, then for any non-empty open set $U\subseteq X$, we have
	$$
	h_{top}^{UC}(T,G_{K}\cap U)=h_{top}^{UC}(T,G_{K}^{T}\cap U)=h_{top}^{UC}(T,G_{K})=h_{top}^{UC}(T,G_{K}^{T})=h_{top}(T,X).$$ 
\end{itemize}
\end{theorem}
\begin{remark}
	We propose a question. Is there a transitive system $(X,T)$ with positive entropy such that $h_{top}^{UC}(T,U)=0$ for some open set $U\subset X$?
\end{remark}

In order to estimate packing entropy of $G_{K}$ and $G_{K}^{T}$, we need that $K$ is convex.
\begin{theorem}\label{packing-entropy}
Suppose that $(X,T)$ satisfies the almost product property. If $K\subseteq \mathcal{M}(X,T)$ is a nonempty compact convex set, then 
\begin{itemize}
	\item for any non-empty open set $U\subseteq X$ with $U\cap C_T(X)\neq\emptyset$, we have 
	\begin{equation*}
		\begin{split}
			&h_{top}^{P}(T,G_{K}\cap U\cap \{x\in X:C_T(X)\subset \omega_T(x)\})\\
			=&h_{top}^{P}(T,G_{K}\cap \{x\in X:C_T(X)\subset \omega_T(x)\})\\
			=&h_{top}^{P}(T,G_{K}\cap U)=h_{top}^{P}(T,G_{K})=\sup\{h(T,\mu):\mu\in K\};
		\end{split}
	\end{equation*}
	\item in particular, if further there is an invariant measure with full support, then for any non-empty open set $U\subseteq X$, we have
	$$h_{top}^{P}(T,G_{K}\cap U)=h_{top}^{P}(T,G_{K}^{T}\cap U)=h_{top}^{P}(T,G_{K})=h_{top}^{P}(T,G_{K}^{T})=\sup\{h(T,\mu):\mu\in K\}.$$
\end{itemize}
\end{theorem}
\begin{remark}
	If $(X,T)$ satisfies specification property, one has $C_T(X)=X.$ However, it may happen that $C_T(X)\neq X$ when $(X,T)$ merely satisfies the almost product property. See Lemma 8.2 and Example 8.3 in \cite{KKO} for examples.
	
	On the other hand, there are several important classes of TDS $($e.g. complete positive entropy systems$)$ admitting full support measure. See \cite[Corollary 7]{Blanc1992} for details. Hence, our hypothesis $U\cap C_T(X)\neq\emptyset$ is essisential.
\end{remark}
\begin{remark}
	(1) For any $x\in X$ and $\mu\in M_{x},$ one has $S_{\mu}\subset \omega_T(x).$ Then we have $\overline{\cup_{\mu\in K}S_{\mu}}\subset\omega_T(x)$ for any $x\in G_K.$ However, this does not imply $C_T(X)\subset \omega_T(x).$
	
	(2) From Theorem \ref{entropy} and \ref{packing-entropy}, local complexity is the same as global complexity from the viewpoint of entropies when $(X,T)$ satisfies the almost product property.
\end{remark}


In \cite[Theorem 1.1]{ZCC2012} the authors obtained $h_{top}^{P}(T,G_{K})=\sup\{h(T,\mu):\mu\in K\}$ under the assumption that $(X,T)$ has specification property and expansiveness. In Theorem \ref{entropy} and \ref{packing-entropy}, we don't need any expansiveness-like property. The proofs of Theorem \ref{entropy} and \ref{packing-entropy} are constructive and are provided in Section \ref{sec:proofentropy} and Section \ref{sec:proofentropy2}, respectively.

\subsection{Application}
In this application section, we will concentrate on the Anosov diffeomorphism. Let's recall their concepts here.
Let $M$ be a compact smooth Riemann manifold without boundary. A diffeomorphism $f:M\to M$ is called an Anosov diffeomorphism if for any $x\in M$ there is a splitting of the tangent space $T_{x}M=E^{s}(x)\oplus E^{u}(x)$ which is preserved by the differential $Df$ of $f$:
\begin{equation*}
	Df(E^{s}(x))=E^{s}(f(x)),\ Df(E^{u}(x))=E^{u}(f(x)),
\end{equation*}
and there are constants $C>0$ and $0<\lambda<1$ such that for all $n\geq 0,$
\begin{equation*}
	|Df^{n}(v)|\leq C\lambda^{n}|v|,\ \forall x\in M,\ v\in E^{s}(x),
\end{equation*}
\begin{equation*}
	|Df^{-n}(v)|\leq C\lambda^{n}|v|,\ \forall x\in M,\ v\in E^{u}(x).
\end{equation*} 
It is well known that transitive Anosov diffeomorphism has specification property, which is stronger than almost product property. Therefore, Theorem \ref{entropy} and \ref{packing-entropy} are applicable in this setting. We will use Theorem \ref{entropy} and \ref{packing-entropy} to quantify topological complexity of optimal orbits and equilibrium states, from both entropy and chaos viewpoints.  In this setting, the measure center $C_T(X)=X$.

\subsubsection{Optimal orbits.}

Let us introduce the concept of optimal orbit as follows. Such concept was given by Ott and Hunt \cite{OT}, Yuan and Hunt \cite{YH}.
\begin{itemize}
\item given a continuous function $f:X\to\mathbb{R}$ taking value into real numbers. Let the Birkhoff sum
$$
f^{(n)}(x):=\sum_{i=0}^{n-1}f(T^{i}x)
$$
and let $\langle f\rangle(x)=\lim_{n\to\infty}\frac{1}{n}f^{(n)}(x)$ if the limit exists. By Birkhoff's ergodic theorem, for every $\mu\in\mathcal{M}(X,T)$, $\mu$-a.e.-$x$, $\langle f\rangle(x)$ is well defined (is $\int fd\mu$, whenever $\mu\in\mathcal{M}^{e}(X,T)$).
If $\langle f\rangle(x_{0})$ is defined, and for each $x\in X$ with $\langle f\rangle(x)$ exists one has $\langle f\rangle(x_{0})\geq \lim_{n\to\infty}\frac{1}{n}f^{(n)}(x)$, then the orbit $\text{orb}(x_0,T)$ is called an \emph{$f$-optimal orbit} of Birkhoff average. Let $$\beta(f):=\sup_{x\in X}\limsup_{n\to\infty}\frac{1}{n}f^{(n)}(x).$$ Due to the Birkhoff's ergodic theorem, we will have
\begin{equation*}
	\beta(f)=\sup_{\mu\in \cM(X,T)}\int fd\mu. 
\end{equation*}
We say $\mu\in \cM(X,T)$ is a $f$-maximizing measure if $\int fd\mu=\beta(f).$
\item given a continuous function $F:X\to Mat(d,\mathbb{R})$ taking values into the space of $d\times d$ real valued matrix. Let
$$
F^{(n)}(x)=F(T^{n-1}x)\cdots F(Tx)\cdot F(x).
$$
The triple $(T,X,F)$ is a linear cocycle. The top Lyapunov exponent at $x\in X$ is defined as the limit
$$
\chi_{max}(F,x):=\lim_{n\to\infty}\frac{1}{n}\log\|F^{n}(x)\|
$$
if it exists. By the Kingman's subadditive ergodic theorem, for every $\mu\in\mathcal{M}(X,T)$, $\mu$-a.e-$x$, $\chi_{max}(F,x)$ is well defined (is a constant, whenever $\mu\in\mathcal{M}^{e}(X,T)$). Let $\chi_{max}(F,\mu):=\int\chi_{max}(F,x)d\mu(x)$.
Analogously, if $\chi_{max}(F,x_{0})$ is defined, and for each $x\in X$ with $\chi_{max}(F,x)$ exists one has $\chi_{max}(F,x_{0})$ $\geq\chi_{max}(F,x)$, then the orbit $\text{orb}(x_0,T)$ is called an \emph{$F$-optimal orbit} of top Lyapunov exponent. 
\end{itemize}
When $d=1,$ the above two situations are coincident. In both situations, define by
\begin{equation}\label{equ:setofoptimalorbit}
  S^{op}_{f}(\mbox{resp.}~S^{op}_{F}):=\{x\in X: \text{orb}(x,T)~\mbox{is an $f$(resp.~$F$)-optimal orbit}\},
\end{equation}
\emph{the set of initial values of optimal orbits}. 
\\
\\
\noindent\textbf{Complexity of optimal orbits from the viewpoint of entropies.} In this section, we describe complexity of optimal orbits from the viewpoint of entropies. Firstly, we consider the upper capacity entropy of $S_{F}^{op}.$
\begin{theorem}\label{theo:uppercapacityoptimalorbit}
	Suppose that $(X,T)$ is a transitive Anosov diffeomorphism on a compact manifold. Let $U$ be a non-empty open set of $X.$ Let the function $F:X\rightarrow GL(m,\mathbb{R})$ be a H$\ddot{\text{o}}$lder continuous matrix function, then
	\begin{equation}\label{equmatrxifullentropy}
		h_{top}^{UC}(T,S_{F}^{op}\cap U)=h_{top}^{UC}(T,S_{F}^{op})=h_{top}^{UC}(T,X)=h_{top}(T,X)>0.
	\end{equation}
\end{theorem}

\medskip

Next, we consider a more refined description of the topological structure of $S_{f}^{op}$ than Theorem \ref{theo:uppercapacityoptimalorbit} for real valued function $f:X\to\mathbb{R}$. And we will show such description of $S_{f}^{op}$ is sensitive to regularity of $f$. To be more specific, we say a point $x\in M$ is \emph{measure-recurrent}, if $x$ generates an invariant measure $\mu\in \cM(X,T)$, i.e., $\mu=\lim_{n\to\infty}\frac{1}{n}\sum_{i=0}^{n-1}\delta_{T^{i}x}\in\cM(X,T)$, and $x$ lies in the support of $\mu$.
If the point $x$ is measure-recurrent, then $\omega_T(x)=S_\mu.$
Define \emph{the set of initial values of measure-recurrent optimal orbits} by 
\begin{equation}\label{equ:measurerecurrence}
  S_{f}^{MR}:=\{x\in X: \text{orb}(x,T)~\mbox{is a measure-recurrent $f$-optimal orbit}\}\subset S^{op}_{f}.
\end{equation}
In parallel, define \emph{the set of initial values of topologically transitive optimal orbits} by 
\begin{equation}\label{equ:topologicalrecurrence}
S_{f}^{TR}:=\{x\in X: \text{orb}(x,T)~\mbox{is a topologically transive $f$-optimal orbit}\}\subset S_{f}^{op}.
\end{equation}
\begin{theorem}\label{theo:regularity sensitivity result}
Suppose that $(X,T)$ is a transitive Anosov diffeomorphism on a compact manifold. 
Then 
\begin{itemize}
  \item for any continuous function $f$ and any non-empty open set $U\subseteq X,$ we have $$h_{top}^{UC}(T,S^{op}_{f}\cap U)=h_{top}^{UC}(T,S^{op}_{f})=h_{top}^{UC}(T,S_{f}^{TR}\cap U)=h_{top}^{UC}(T,S_{f}^{TR})=h_{top}(T,X)>0.$$
  \item there exists a Baire generic subset $\mathcal{F}$ in the space of continuous functions and an open and dense subset $\mathcal{G}$ in the space of H\"{o}lder continuous or $C^{1}$ smooth functions such that 
  \begin{itemize}
	\item for any $f\in \mathcal{F}$ or $f\in \mathcal{G}$ and any non-empty open set $U\subseteq X,$ we have  $$h_{top}^{B}(T,S^{op}_{f}\cap U)=h_{top}^{B}(T,S^{op}_{f})=h_{top}^{P}(T,S^{op}_{f}\cap U)=h_{top}^{P}(T,S^{op}_{f})=0.$$
	\item for any $f\in \mathcal{F},$ we have $S_{f}^{TR}=S_{f}^{MR}=S^{op}_{f}\neq\emptyset.$
	\item for any $f\in \mathcal{G}$ and any non-empty open set $U\subseteq X,$ we have 
	\begin{itemize}
		\item $S_{f}^{TR}\cap S_{f}^{MR}=\emptyset.$ Moreover, $S_{f}^{TR}\neq\emptyset$ and there is a unique periodic orbit $\mathrm{orb}(x_0,T)$ such that $S_{f}^{MR}=\cup_{i=0}^{\infty}T^{-i}\mathrm{orb}(x_0,T);$
		\item $h_{top}^{UC}(T,S_{f}^{MR}\cap U)=h_{top}^{UC}(T,S_{f}^{MR})=0.$
	\end{itemize}
  \end{itemize}
  \item there exists a dense subset $\mathcal{H}$ in the space of continuous functions such that for any $f\in\mathcal{H}$ and any non-empty open set $U\subseteq X,$ we have $$h_{top}^{B}(T,S_{f}^{MR}\cap U)>0,\ h_{top}^{P}(T,S_{f}^{MR}\cap U)>0,\ h_{top}^{B}(T,S_{f}^{TR}\cap U)>0,\ h_{top}^{P}(T,S_{f}^{TR}\cap U)>0.$$
\end{itemize}
\end{theorem}

\noindent\textbf{Complexity of optimal orbits from the viewpoint of chaos.} Besides entropies, chaos is another tool to describe complexity of fractal sets.
We will provide the following qualitative characterization on the chaos in $S_{f}^{TR}$ and $S_{f}^{MR}.$

\begin{theorem}\label{theorem-chaos}
Suppose that $(X,T)$ is a transitive Anosov diffeomorphism on a compact manifold. Then there exists an open and dense subset $\mathcal{G}$ in the space of H\"{o}lder continuous or $C^{1}$ smooth functions such that for any $f\in\mathcal{G}$ and any non-empty open set $U\subseteq X,$ the $f$-maximizing measure is unique and 
\begin{itemize}
  \item either the $f$-maximizing measure is supported on a periodic point which is not a fixed point and there exists an uncountable DC1-scrambled set $S\subseteq S_{f}^{TR}\cap U,$
  \item or the $f$-maximizing measure is supported on a fixed point and there exists an uncountable set $S\subseteq S_{f}^{TR}\cap U$ such that $S$ is chaotic in the following sense\footnote{By definition of Li-Yorke chaos (see Section \ref{chaos}), (\ref{equation-DE}) implies Li-Yorke chaos.}: for any $x,y\in S$ and any $t>0,$
  \begin{equation}\label{equation-DE}
  	\limsup_{n\to +\infty}d(T^nx,T^ny)
  	>0,\ \limsup_{n\to +\infty}\frac{1}{n}\sharp \left\{i:d(T^{i}x,T^{i}y)<t,0\leq i \leq n-1\right\}=1.
  \end{equation}
\end{itemize}
But there is no Li-Yorke pair in $S_{f}^{MR}$.
\end{theorem}

\subsubsection{Equilibrium states}
Other than the application in optimal orbits, we will also apply our main results on studying the equilibrium states arising in thermodynamics formalism. For a TDS $(X,T)$ and a continuous function $f:X\to\mathbb{R}$, we define by $P(T,f)$ the \emph{topological pressure} (See section \ref{sec:somedefintions} for the precise definitions), and there is a \emph{variational principle}:
$$
P(T,f)=\sup_{\nu\in\mathcal{M}(X,T)}\left\{h(T,\nu)+\int fd\nu\right\}.
$$
A measure $\mu\in\mathcal{M}(X,T)$ is said to be an \emph{$f$-equilibrium state}, if $h(T,\mu)+\int fd\mu=P(T,f)$. Deducing from Theorem \ref{entropy}, \ref{packing-entropy} and \cite[Lemma 3.11]{HTW}, we obtain the following result about the complexity of equilibrium states.
\begin{theorem}\label{theorem-pointwisedimension}
Suppose that $(X,T)$ is a transitive Anosov diffeomorphism on a compact manifold. Let $f:X\to\mathbb{R}$ be a continuous function, and $\mathfrak{P}(f)$ be the set of all $f$-equilibrium states. Then for any $\mu\in\mathfrak{P}(f)$ and any non-empty open set $U\subseteq X,$ we have
\begin{itemize}
  \item $h_{top}^{UC}(T,G_{\mu}^{T}\cap U)=h_{top}^{UC}(T,G_{\mu}\cap U)=h_{top}^{UC}(T,G_{\mu}^{T})=h_{top}^{UC}(T,G_{\mu})=h_{top}(T,X)>0$;
  \item $h_{top}^{P}(T,G_{\mu}^{T}\cap U)=h_{top}^{P}(T,G_{\mu}\cap U)=h_{top}^{P}(T,G_{\mu}^{T})=h_{top}^{P}(T,G_{\mu})=h(T,\mu)$;
  \item $h_{top}^{B}(T,G_{\mu}^{T}\cap U)=h_{top}^{B}(T,G_{\mu}\cap U)=h_{top}^{B}(T,G_{\mu}^{T})=h_{top}^{B}(T,G_{\mu})=h(T,\mu)$.
\end{itemize}	
\end{theorem}
Based on Theorem \ref{theorem-pointwisedimension}, we obtain that the three entropies of the genertic point $G_{\mu}$ for the $f$-equilibrium state are different. To be more specific, if $f\equiv0$ (subject to a $C^{0}$-coboundary), we have that $\mathfrak{P}(0)$ is a singleton, and $h(T,\mu)=h_{top}(X,T)$ for $\mu\in\mathfrak{P}(0)$. Therefore, it follows from Theorem \ref{theorem-pointwisedimension} that
$$
h_{top}^{B}(T,G_{\mu}^{T})=h_{top}^{B}(T,G_{\mu})=h_{top}^{P}(T,G_{\mu}^{T})=h_{top}^{P}(T,G_{\mu})=h_{top}^{UC}(T,G_{\mu}^{T})=h_{top}^{UC}(T,G_{\mu})=h_{top}(T,X).
$$
When $f$ is H\"{o}lder continuous and is not cohomologous to a constant, then by \cite[Proposition 20.3.10]{KatHas} one has $h(T,\mu)<h_{top}(X,T)$ for any $\mu\in\mathfrak{P}(f)$. Therefore, it follows from Theorem \ref{theorem-pointwisedimension} that
$$
h_{top}^{B}(T,G_{\mu}^{T})=h_{top}^{B}(T,G_{\mu})=h_{top}^{P}(T,G_{\mu}^{T})=h_{top}^{P}(T,G_{\mu})<h_{top}^{UC}(T,G_{\mu}^{T})=h_{top}^{UC}(T,G_{\mu})=h_{top}(T,X).
$$

\medskip


The paper is organized as follows. In preliminary Section \ref{sec:somedefintions}, we recall some definitions (e.g. upper capacity topological entropy, almost product property, distributional chaos, etc.) In Section \ref{sec:proofentropy}, we consider upper capacity entropy formula and give the proof of Theorem \ref{entropy}. In Section \ref{sec:proofentropy2}, we consider packing entropy formula and give the proof of Theorem \ref{packing-entropy}. In Section \ref{sec:proofcocycle}, we study maximal Lyapunov exponents of cocycles and prove Theorem \ref{theo:uppercapacityoptimalorbit}.  As applications, in Section \ref{sec:regularity sensitivity result} we use Theorem \ref{entropy} and \ref{packing-entropy} to quantify topological complexity of optimal orbits, and we provide the existence of chaos in $S_{f}^{TR}.$

\section{Preliminary}\label{sec:somedefintions}
The context will be composed three parts, namely, the entropies of subsets, almost product property, distributional chaos and scrambled set.
\subsection{Entropies for subsets}
Let us begin with three types of topological entropies for subsets.
\begin{definition}[Topological pressure and Bowen topological entropy]
	Let $Z\subseteq X$, $s\geq 0$, $n\in\mathbb{N}$, and $\varepsilon >0$, and a continuous function $f:X\to\mathbb{R}$, define
	$$
\mathcal{M}_{N,\varepsilon}^{s}(Z,f):=\inf \sum_{i}\exp\left(-sn_{i}+f^{(n_{i})}(x_{i})\right),
    $$
	where the infimum is taken over all finite or countable families $\{B_{n_{i}}(x_{i},\varepsilon)\}$ such that $x_{i}\in X$, $n_{i}\geq N$ and $\bigcup_{i}B_{n_{i}}(x_{i},\varepsilon)\supseteq Z$, where
    $$
    B_{n}(x,\epsilon):=\{y\in X:d(T^{i}x,T^{i}y)<\epsilon,0\leq i\leq n-1\}.
    $$
    The quantity $\mathcal{M}_{N,\varepsilon}^{s}(Z)$
	does not decrease as $N$ increases and $\varepsilon$
	decreases, hence the following limits exist:
	$$\mathcal{M}_{\varepsilon}^{s}(Z,f)=\lim_{N\to \infty}\mathcal{M}_{N,\varepsilon}^{s}(Z,f)$$
	and
	$$\mathcal{M}^{s}(Z,f)=\lim_{\varepsilon\to 0}\mathcal{M}_{\varepsilon}^{s}(Z,f).$$
	\emph{The topological pressure} $P(T,Z,f)$
	can be equivalently defined as a critical value of the
	parameter $S$, where $\mathcal{M}^{s}(Z)$ jumps from $\infty$ to $0$, i.e.
	\begin{equation}
	\mathcal{M}^{s}(Z):=
	\begin{cases}
	0,&s>P(T,Z,f),\\
	\infty,&s<P(T,Z,f).
	\end{cases}
	\end{equation}
In particular case $f\equiv0$, this definition yields the \emph{Bowen topological entropy} by
\begin{equation}\label{equ:bowen's topological entropy}
  h_{top}^{B}(T,Z):=P(T,Z,0).
\end{equation}
\end{definition}
When $Z=X$, we can compute the topological pressure as follows. For each $n\in\mathbb{N}$ and $\epsilon>0$, let
$$
\mathcal{Z}_{n}(f,\epsilon):=\inf\left\{\sum_{x\in E}\exp\left(\sup_{y\in B_{n}(x,\epsilon)}f^{(n)}(y)\right):B_{n}(E,\epsilon)=X\right\},
$$
where $B_{n}(E,\epsilon):=\bigcup_{x\in E}B_{n}(x,\epsilon)$, and the topological pressure is given by
$$
P(T,f,X):=P(f,X)=\lim\limits_{\epsilon\to0}\limsup\limits_{n\to\infty}\frac{1}{n}\log\mathcal{Z}_{n}(f,\epsilon)=\sup_{\nu\in\mathcal{M}(X,T)}\left\{h(T,\nu)+\int fd\nu\right\}.
$$


\begin{definition}[Packing topological entropy]
Let $Z\subseteq X$. For $s\geq 0$, $n\in\mathbb{N}$, and $\epsilon >0$, define
$$P_{N,\epsilon}^{s}(Z):=\sup \sum_{i}\exp(-sn_{i}),$$
where the supremum is taken over all finite or countable pairwise disjoint families $\{\overline{B}_{n_{i}}(x_{i},\epsilon)\}$ such that $x_{i}\in Z$, $n_{i}\geq N$ for all $i$, where
$$\overline{B}_{n}(x,\epsilon):=\{y\in X:d_{n}(x,y)\leq\epsilon\}$$ and $$d_{n}(x,y):=\max\{d(T^{k}x,T^{k}y):k=0,\dots,n-1\}.$$
The quantity $P_{N,\epsilon}^{s}(Z)$
does not decrease as $N,$ $\epsilon$
decrease, hence the following limits exist:
$$
P_{\epsilon}^{s}(Z)=\lim_{N\to \infty}P_{N,\epsilon}^{s}(Z).
$$
Define
$$\mathcal{P}_{\epsilon}^{s}(Z):=\inf\left\{\sum_{i=1}^{\infty}P_{\epsilon}^{s}(Z_{i}):\bigcup_{i=1}^{\infty}Z_{i}\supseteq Z\right\}.$$
Clearly, $\mathcal{P}_{\epsilon}^{s}(Z)$
satisfies the following property: if $Z \subseteq \bigcup_{i=1}^{\infty}Z_{i}$, then $\mathcal{P}_{\epsilon}^{s}(Z)\leq \sum_{i=1}^{\infty}\mathcal{P}_{\epsilon}^{s}(Z_{i})$. There exists a critical value of the parameter $s$, which we will denote by $h_{top}^{P}(T,Z,\epsilon)$, where $\mathcal{P}_{\epsilon}^{s}(Z)$ jumps from $\infty$ to $0$, i.e.
    \begin{equation}
	\mathcal{P}^{s}_{\epsilon}(Z):=
	\begin{cases}
	0,&s>h_{top}^{P}(T,Z,\epsilon),\\
	\infty,&s<h_{top}^{P}(T,Z,\epsilon).
	\end{cases}
	\end{equation}
	Note that $h_{top}^{P}(T,Z,\epsilon)$ increases when $\epsilon$ decreases, \emph{the packing topological entropy of $Z$} is defined by
	$$h_{top}^{P}(T,Z):=\lim_{\epsilon\to 0}h_{top}^{P}(T,Z,\epsilon).$$
\end{definition}

\begin{definition}[Upper capacity topological entropy]
Let $Z\subseteq X$ be a non-empty set. For $\epsilon>0$, a set $E\subseteq Z$ is called an $(n,\epsilon)$-separated set of $Z$ if $x,y\in E$, $x\neq y$ implies $d_{n}(x,y)>\epsilon$. Let $s_{n}(Z,\epsilon)$ denote the largest cardinality of $(n,\epsilon)$-separated sets for $Z$, \emph{the upper capacity topological
entropy of $Z$} is defined by
$$
h_{top}^{UC}(T,Z):=\lim_{\epsilon\rightarrow 0}\limsup_{n \rightarrow \infty}\frac{1}{n}\log s_{n}(Z,\epsilon).
$$
\end{definition}

Some basic properties of $h_{top}^{B},h_{top}^{P}$ and $h_{top}^{UC}$ are as following:
\begin{proposition}\cite[Proposition 2.1]{Feng-Huang}\label{basic properties}
	\begin{description}
		\item[(1)] For $Z \subseteq Z^{\prime}$
			$$h_{top}^{UC}(T,Z) \leqslant h_{top}^{UC}\left(T, Z^{\prime}\right),\ h_{top}^{B}(T, Z) \leqslant h_{top}^{B}\left(T, Z^{\prime}\right),\ h_{top}^{P}(T, Z) \leqslant h_{top}^{P}\left(T, Z^{\prime}\right).$$
			\item[(2)] For any $Z\subset X$, $h_{top}^{B}(T,Z)\leq h_{top}^{P}(T,Z)\leq h_{top}^{UC}(T,Z).$
			\item[(3)] If $Z$ is $T$-invariant and compact, then $h_{top}^{B}(T,Z)= h_{top}^{P}(T,Z)= h_{top}^{UC}(T,Z)$.
			\item[(4)] For $Z \subseteq \bigcup_{i=1}^{\infty} Z_{i},$ $h_{top}^{B}(T, Z) \leqslant \sup _{i \geqslant 1} h_{\top}^{B}\left(T, Z_{i}\right), \ h_{top}^{P}(T, Z) \leqslant \sup _{i \geqslant 1} h_{top}^{P}\left(T, Z_{i}\right).$
	\end{description}
\end{proposition}
Before we state more properties on entropies, let $\Omega(T,X)$ be the set of all non-wandering points where a point $x\in X$ is non-wandering if for every neighborhood $U$ of $x$ there exists $K\in\mathbb{N^{+}}$ and $y\in U$ such that $T^ky\in U$, and let $Rec(T,X):=\{x\in X:x\in \omega_T(x)\}$ be the set of recurrent points. 
\begin{proposition}\label{prop-AA}
	Let $Z\subseteq X$ be a non-empty set satisfying that $\Omega(X,T)\subseteq\overline{Z}.$ Then $h_{top}^{UC}(T,Z)=h_{top}^{UC}(T,X)=h_{top}(T,X).$ 
\end{proposition}
\begin{proof}
	For each $\eta>0,$ by the definition of $h_{top}^{UC}$ on $\Omega(T,X)$ there exists $\epsilon_{*}>0$ such that for any $0<\epsilon<\epsilon_{*}$ there exists $N_{\epsilon}\in\mathbb{N}$ such that for any integer $n>N_{\epsilon}$, one has $\frac{1}{n}\log s_{n}(\Omega(T,X),\epsilon)>h_{top}^{UC}(T,\Omega(T,X))-2\eta.$ Let $E_n\subset \Omega(T,X)$ be a $(n,\epsilon)$-separated set of $\Omega(T,X)$ with $\sharp E_n=s_{n}(\Omega(T,X),\epsilon).$ Using the hypothesis that $\Omega(T,X)\subseteq\overline{Z}$, for any $y\in E_n$ there exists $x(y)\in Z$ such that $d_n(x(y),y)<\epsilon/3.$ Then $\tilde{E}_n:=\{x(y):y\in E_n\}\subset Z$ is a $(n,\epsilon/3)$-separated set of $Z$ with $\sharp\tilde{E}_n=\sharp E_n=s_{n}(\Omega(T,X),\epsilon)$ which means $s_{n}(Z,\epsilon/3)\geq s_{n}(\Omega(T,X),\epsilon).$ Thus $h_{top}^{UC}(T,Z)\geq h_{top}^{UC}(T,\Omega(T,X))-2\eta.$ By the arbitrariness of $\eta$, this directly means $h_{top}^{UC}(T,Z)\geq h_{top}(T,\Omega(T,X)).$ Combining with $h_{top}(T,\Omega(X,T))=h_{top}(T,X),$ one has $h_{top}^{UC}(T,Z)=h_{top}(T,X).$
\end{proof}

\begin{remark}
	$(1)$ Proposition \ref{prop-AA} does not hold for Bowen's topological entropy and packing entropy. For example, suppose that $(X,T)$ satisfies the almost product property and  there is an invariant measure with full support. If $\mu\in\mathcal{M}(T,X)$ satisfies $h(T,\mu)<h_{top}(T,X),$ then by \cite[Theorem 1.4]{HTW} and Theorem \ref{packing-entropy}, $h_{top}^{B}(T,G_\mu^{T})=h_{top}^{P}(T,G_\mu^{T})=h(T,\mu)<h_{top}(T,X).$ Since $G_\mu^{T}$ is dense in $X$, Proposition \ref{prop-AA} does not hold for Bowen's topological entropy and packing entropy. 
	
	$(2)$ If $\Omega(T,X)=X$, then $\overline{Rec}(T,X)=X.$ So $h_{top}^{UC}(T,Rec(T,X))=h_{top}(T,X).$
	 
	$(3)$ If $(X,T)$ is transitive, then $\overline{Tran}=X$. So $h_{top}^{UC}(T,Tran)=h_{top}(T,X).$ Moreover, the upper capacity topological entropy of any transitive orbit equals to the topological entropy of the system.
\end{remark}
\begin{remark}
	Suppose that $(X,T)$ satisfies the almost product property and there is an invariant measure with full support. Let $U\subseteq X$ be a non-empty open set.  For any  $\mu\in\mathcal{M}(T,X),$ by \cite[Theorem 1.4]{HTW} one has $h_{top}^{B}(T,G^{T}_{\mu})=h(T,\mu).$ Combining with \cite[Lemma 3.11]{HTW}, one has $h_{top}^{B}(T,G^{T}_{\mu}\cap U)=h_{top}^{B}(T,G^{T}_{\mu})=h(T,\mu).$  So $h_{top}^{B}(T,U)=h_{top}(T,X)$ by the arbitrariness of $\mu.$ By Proposition \ref{basic properties}(2),
	\begin{equation}\label{equation-DD}
		h_{top}^{B}(T,U)=h_{top}^{P}(T,U)=h_{top}^{UC}(T,U)=h_{top}(T,X).
	\end{equation} 
\end{remark}

\begin{lemma}\label{packing-lemma}
	Given a TDS $(X,T)$, let $B \subseteq X$ be invariant and $U \subseteq X$ be a nonempty open set. Then $h_{top }^{P}(T, B \cap Tran \cap U)=h_{top}^{P}(T, B \cap Tran).$
\end{lemma}
\begin{proof}
	By Proposition \ref{basic properties}(1), $h_{top }^{P}(T, B \cap Tran \cap U)\leq h_{top}^{P}(T, B \cap Tran)$ is obvious. Now we start to consider the reverse direction.
	Let $B^{T}=B \cap Tran$. Form the definition of packing entropy, one has $h_{top }^{P}(T, Z)=h_{top}^{P}(T,TZ)$ for any $Z\subseteq X.$ Notice that by Proposition \ref{basic properties}(1) for any $n \geq 1,$
	\begin{equation}\label{equation-DA}
		h_{top}^P\left(T, T^{-n}\left(U \cap B^{T}\right)\right)=h_{top}^P\left(T, T^{n} T^{-n}\left(U \cap B^{T}\right)\right) \leq h_{top}^P\left(T, U \cap B^{T}\right),
	\end{equation}
	and by the definition of transitivity and invariance of $B$ and $Tran$
	\begin{equation}\label{equation-DB}
		B^{T}=B^{T} \cap\left(\bigcup_{n \geqslant 0} T^{-n} U\right) \subseteq \bigcup_{n \geqslant 0} T^{-n}\left(U \cap B^{T}\right).
	\end{equation}
	Thus one has 
	\begin{align*}
		h_{top}^P\left(T, B^{T}\right) &\leq h_{top}^P\left(T, \bigcup_{n \geq 0} T^{-n}\left(U \cap B^{T}\right)\right)\quad &(\text{using}\ (\ref{equation-DB}))\\
		&\leq \sup _{n \geq 0} h_{top}^P\left(T, T^{-n}\left(U \cap B^{T}\right)\right)\quad &(\text{using Proposition}\ \ref{basic properties}(1))\\
		& \leq h_{top}^P\left(T, U \cap B^{T}\right),\quad &(\text{using}\ (\ref{equation-DA}))
	\end{align*}
	as we want.
\end{proof}
\subsection{Almost product property}

Almost product property was first introduced by Pister and Sullivan in \cite{PS}, and it turn out to be a (strictly) weaker concept than specification property. For example, it is well known (e.g. \cite{PS,Tho2012}) that almost product property holds for every $\beta$-shift, however the specification property is not satisfied for every parameter $\beta$. Let us recall the definition of almost product property as follows.
\begin{definition}[Blowup function]
	Let $g:\mathbb{N}\rightarrow\mathbb{N}$ be a given nondecreasing unbounded map with the properties
	\begin{equation}
	g(n)<n \ \mathrm{and} \ \lim_{n\rightarrow \infty}\frac{g(n)}{n}=0.
	\end{equation}
	The function $g$ is called the \emph{blowup function}. Let $x\in X$ and $\epsilon>0$. The $g$-blowup of $B_n(x,\epsilon)$ is the closed set
	\begin{equation*}
	B_{n}(g;x,\epsilon):=\{y \in X : \exists \Lambda \subseteq \Lambda_{n},\sharp(\Lambda_{n} \backslash\Lambda)\leq g(n) \ \mathrm{and}\ \max_{j \in\Lambda}\{d(T^{j}x,T^{j}x) \}\leq \epsilon\},
	\end{equation*}
	where $\Lambda_{n}: = \{0,\dots ,n-1\}$.
\end{definition}

\begin{definition}[Almost product property]\label{definition-AA}
	We say that $(X,T)$ satisfies the \emph{almost product property}, if there is a blowup function $g$ and a nonincreasing function $m:\mathbb{R}^{+}\rightarrow \mathbb{N}$, such
	that for any $k\geq 2$, any $k$ points $x_{1},\dots,x_{k} \in X$, any positive $\epsilon_{1},\dots,\epsilon_{k}$ and any integers
	$n_{1}\geq m(\epsilon_{1}),\dots,n_{k}\geq m(\epsilon_{k})$,
	\begin{equation}
	\bigcap_{j=1}^{k}T^{-M_{j-1}}B_{n_{j}}(g;x_{j},\epsilon _{j})\neq \emptyset,
	\end{equation}
	where $M_{0}:=0,M_{i}:=n_{1}+\dots n_{i},i=1,2,\dots,k-1$.
\end{definition}

The almost periodic points plays important role in studying the measure center of $(X,T)$ satisfying almost product property. Here a point $x \in X$ is \emph{almost periodic}, if for any open neighborhood $U$ of $x$, there exists $N \in \mathbb{N}$ such that $f^k(x) \in U$ for some $k \in [n, n+N]$ for every $n \in \mathbb{N}$. It is well-known that $x$ {is almost periodic}, if and only if  $\overline{\text{orb}(x,T)}$ is a minimal set. 
Let AP denote the set of almost periodic points.
\begin{proposition}\label{measure-center}
	Suppose that $(X,T)$ satisfies almost product property. Then $C_T(X)=\overline{AP}.$
\end{proposition}
\begin{proof}
	For any $x\in AP$ and $\mu\in \mathcal{M}(\overline{\text{orb}(x,T)},T),$ we have $x\in S_\mu.$ Then $\overline{AP}\subset C_T(X).$ To see the reverse direction, by \cite[Lemma 3.4]{HTW}, if $(X,T)$ has almost product property, then ergodic measures supported
	on minimal sets are dense in $\mathcal{M}(X,T).$ Thus for any $\mu\in\mathcal{M}(X,T),$ there exists $\{\mu_i\}_{i=1}^{\infty}\subset\mathcal{M}(X,T)$ such that $S_{\mu_i}$ is minimal for any $i\in\mathbb{N^{+}}$ and $\lim_{i\to \infty}\mu_i=\mu$ which means $\mu(\overline{\cup_{i=1}^{\infty}S_{\mu_i}})\geq \limsup_{j\to\infty}\mu_j(\overline{\cup_{i=1}^{\infty}S_{\mu_i}})=1.$ In other words, $S_\mu\subset \overline{\cup_{i=1}^{\infty}S_{\mu_i}}\subset\overline{AP}.$  So $C_T(X)=\overline{AP}.$
\end{proof}

\subsection{Distributional chaos and scrambled set}\label{chaos}
A pair of points $x,y$ is said to be a \emph{Li-Yorke pair} if
$$\liminf_{n\to +\infty}d(T^nx,T^ny)=0,\ \limsup_{n\to +\infty}d(T^nx,T^ny)
>0.
$$
Given a TDS $(X,T)$, a subset $S\subseteq X$ is called a \emph{scrambled set} if any pair of distinct two points $x,y$ of $S$ is a Li-Yorke pair, and $(X,T)$ is \emph{Li-Yorke chaotic} if there exists an uncountable scrambled set.

Distributional chaos \cite{SS1994} is a refinement of Li-Yorke chaos. There are three variants of distributional chaos: DC1 (distributional chaos of type 1), DC2 and DC3 (ordered from strongest to weakest). We will focus on DC1, and we refer to \cite{Dwic,SS,SS2} and references therein for more detailed information on distributional chaos theory.

For any positive integer $n$, points $x,y \in X$ and $t \in \mathbb{R}$ let
\begin{equation}
	\Phi _{xy}^{(n)}(t,T)=\frac{1}{n}\sharp\{i:d(T^{i}x,T^{i}y)<t,0\leq i \leq n-1\}.
\end{equation}
Let us denote by $\Phi _{xy}$ and $\Phi _{xy}^{*}$ the following functions:
\begin{equation}
	\Phi _{xy}(t,T)=\liminf_{n \to \infty}\Phi _{xy}^{(n)}(t,T),\ \Phi _{xy}^{*}(t,T)=\limsup_{n \to \infty}\Phi _{xy}^{(n)}(t,T).
\end{equation}
Both functions $\Phi _{xy}$ and $\Phi _{xy}^{*}$ are nondecreasing, $\Phi _{xy}(t,T)=\Phi _{xy}^{*}(t,T)=0$ for $t<0$
and $\Phi _{xy}(t)=\Phi _{xy}^{*}(t)=1$ for $t>\mathrm{diam}X$.
A pair $x,y\in X$ is \emph{DC1-scrambled} if the following two conditions hold:
\begin{equation}
	\Phi _{xy}(t_{0},T)=0\ \mathrm{for}\ \mathrm{some}\ t_{0}>0\ \mathrm{and}
\end{equation}
\begin{equation}
	\Phi _{xy}^{*}(t,T)=1\ \mathrm{for}\ \mathrm{all}\ t>0.
\end{equation}
Meanwhile, a set $S$ is called a \emph{DC1-scrambled set} if any pair of  distinct points in $S$ is DC1-scrambled.

\subsection{Saturated set}\label{saturated set}
In this subsection, we show that every saturated set is a Borel set. We first recall the following definitions.
Let $\left\{\varphi_{j}\right\}_{j \in \mathbb{N}}$ be a dense subset of the space of continuous functions, then
$$
d(\xi, \tau):=\sum_{j=1}^{\infty} \frac{\left|\int \varphi_{j} \mathrm{~d} \xi-\int \varphi_{j} \mathrm{~d} \tau\right|}{2^{j+1}\left\|\varphi_{j}\right\|}
$$
defines a metric on $\cM(X)$ for the weak $^{*}$ topology, where $\left\|\varphi_{i}\right\|=\max \left\{\left|\varphi_{i}(x)\right|: x \in X\right\}$. Note that
$$
d(\xi, \tau) \leqslant 1\ \forall \xi, \tau \in \cM(X).
$$
It is well known that the natural projection $x \mapsto \delta_{x}$ is continuous and if we define operator $P_{T}$ on $\cM(X)$ by formula $P_{T}(\mu)(A)=\mu\left(T^{-1}(A)\right)$, then we can identify $(X, T)$ with $P_{T}$ restricted to the set of Dirac measures (these systems are conjugate). Therefore, without loss of generality we will assume that $d(x, y)=d\left(\delta_{x}, \delta_{y}\right)$. Denote a ball in $\cM(X)$ by
$
\mathcal{B}(\nu, \zeta):=\{\mu \in \cM(X): d(\nu, \mu) \leqslant \zeta\}.
$
\begin{proposition}
	Given a TDS $(X,T),$ then for any nonempty compact connected set $K\subseteq \mathcal{M}(X,T),$ $\{x\in X:M_{x}\subset K\},$ $\{x\in X:K\subset M_{x}\}$ and $\{x\in X:K\cap M_{x}\neq\emptyset\}$ are Borel sets. In particular, $G_K$ is a Borel set.
\end{proposition}
\begin{proof}
	First, we show that $\{x\in X:K\subset M_{x}\}$ is a Borel set.
	Since $K$ is compact, there exist open balls $U_i$ in $\cM(X)$ such that
	\begin{description}
		\item[(a)] $\lim\limits_{i\to\infty}\mathrm{diam}(U_i)=0;$
		\item[(b)] $U_i\cap K\neq \emptyset\ \forall i\in\mathbb{N^{+}};$
		\item[(c)] each point of $K$ lies in infinitely many $U_i$.
	\end{description}
	Put $P(U_{i})=\left\{x \in X : M_x \cap U_i \neq \emptyset\right\}.$
	It is easy to see that $\{x\in X:K\subset M_{x}\}=\bigcap_{i=1}^{\infty} P(U_i)$ and $\{x\in X:K\cap M_{x}\neq\emptyset\}=\bigcap_{i=1}^{\infty}\bigcup_{j=i}^{\infty} P(U_j).$ Now for any $i\in\mathbb{N^{+}},$ we assume that $U_i=\{\mu \in \cM(X): d(\nu_i, \mu) < \zeta_i\}$ for some $\nu_i\in \cM(X)$ and some $\zeta_i>0.$ Let $U_i^l=\{\mu \in \cM(X): d(\nu_i, \mu) < (1-\frac{1}{l})\zeta_i\}$ for any $l\in\mathbb{N^{+}},$ then we have 
    \begin{equation*}
    	P(U_{i})=\bigcup_{l=1}^{\infty}\bigcap_{N=0}^{\infty}\bigcup_{n=N}^{\infty}\{x\in X:\frac{1}{n}\sum_{j=0}^{n-1}\delta_{T^{j}x}\in U_i^l\}.
    \end{equation*}
    Since $x\to \frac{1}{n}\sum_{j=0}^{n-1}\delta_{T^{j}x}$ is continuous for fixed $n\in\mathbb{N},$ the sets $\{x\in X:\frac{1}{n}\sum_{j=0}^{n-1}\delta_{T^{j}x}\in U_i^l\}$ are open. So $P(U_{i})$ is a Borel set for any $i\in\mathbb{N^{+}}$ which implies $\{x\in X:K\subset M_{x}\}$ and $\{x\in X:K\cap M_{x}\neq\emptyset\}$ are Borel sets. 
    
    On the other hand, it is easy to see that $\{x\in X:M_{x}\subset K\}=\bigcap_{l=1}^{\infty}\bigcup_{N=0}^{\infty}\bigcap_{n=N}^{\infty}\{x\in X:\frac{1}{n}\sum_{j=0}^{n-1}\delta_{T^{j}x}\in V_l\},$ where $V_l:=\{\mu \in \cM(X): \inf_{\nu\in K} d(\nu, \mu)< \frac{1}{l}\}.$ So $\{x\in X:M_{x}\subset K\}$ is a Borel set which implies $G_K$ is a Borel set. 
\end{proof}

\section{Upper capacity entropy formula}\label{sec:proofentropy}
In this section, we prove Theorem \ref{entropy}. Before proving Theorem \ref{entropy}, we need some basic results which is also useful for the
proof of our main theorems. 
\subsection{Some lemmas}
We first recall the following definitions.
For every $\delta >0,$ $\epsilon >0$ and $n\in\mathbb{N}^{+}$, two points $x$ and $y$ are \emph{$(\delta,n,\epsilon)$-separated} if
\begin{equation*}
	\sharp\{j:d(T^{j}x,T^{j}y)>\epsilon,j=1,2,\cdots,n-1\}\geq \delta n.
\end{equation*}
A subset $E$ is \emph{$(\delta,n,\epsilon)$-separated} if any pair of different points of $E$ are $(\delta,n,\epsilon)$-separated.
For $x \in X$, define
\begin{equation}
	P_{n}(x):=\frac{1}{n}\sum_{j=0}^{n-1}\delta_{T^{j}x}.
\end{equation}
Let $F\subseteq \mathcal{M}(X)$  be a neighbourhood of $\nu \in \cM(X,T)$, and define
$$
X_{n,F}:=\{x\in X:P_{n}(x)\in F\}.
$$
Pister and Sullivan proved in \cite{PS2005} the following proposition.
\begin{proposition}\cite[Proposition 2.1]{PS2005}\label{pro2.1}
	Suppose $\mu \in \cM^{e}(X,T)$ and $\eta >0$. Then there exists $\delta^{*}>0$ and $\epsilon^{*}>0$ so that for each neighbourhood $F$ of $\mu$ in $\cM(X)$, there exists $n^{*}_{F,\mu,\eta}$ such that for any $n \geq n^{*}_{F,\mu,\eta}$, there exists a $(\delta^{*},n,\epsilon^{*})$-separated set $\Gamma_{n} \subseteq X_{n,F}$, such that
	\begin{equation}
		\sharp\Gamma_{n} \geq e^{n(h(T,\mu)-\eta)}.
	\end{equation}
\end{proposition}
On the other hand, for any nonempty compact connected set $K\subseteq \mathcal{M}(X,T)$, one has
\begin{lemma}[\cite{PS}, Page 944]\label{AA}
	There exists a sequence $\{\alpha_{1},\alpha_{2},\dots\}$ in $K$ such that
	\begin{equation}
		\overline{\{\alpha_{j}:j\in \mathbb{N}^{+},j>n\}}=K,\ \forall n \in \mathbb{N}^{+} \ \text{and} \ \lim_{j\rightarrow \infty}d(\alpha_{j},\alpha_{j+1})=0.
	\end{equation}
\end{lemma}

\begin{lemma}[\cite{PS}, Lemma 2.1]\label{ds}
	Suppose that $(X,T)$ satisfies the almost product property. Given $x_{1},\dots,x_{k} \in X$, $\epsilon_{1},\dots,\epsilon_{k}$ and $n_{1}\geq m(\epsilon_{1}),\dots,n_{k}\geq m(\epsilon_{k}).$ Assume that there are $\nu_{j}\in \mathcal{M}(X,T)$ and $\zeta_{j}>0$ satisfying
	\begin{equation}
		P_{n_{j}}(x_{j}) \in \cB(\nu_{j},\zeta_{j}),\ j=1,2,\dots,k.
	\end{equation}
	Then for any $z \in \cap_{j=1}^{k}T^{-Q_{j-1}}B_{n_{j}}(g;x_{j},\epsilon_{j})$ and any probability measure $\alpha\in\mathcal{M}(X),$
	\begin{equation}
		d(P_{Q_{k}}(z),\alpha)\leq \sum_{j=1}^{k}\frac{n_{j}}{Q_{k}}\left(\zeta_{j}+\epsilon_{j}+\frac{g(n_{j})}{n_{j}}+d(\nu_{j},\alpha)\right),
	\end{equation}
	where $Q_{0}=0$, $Q_{i}=n_{1}+\dots+n_{i}$.
\end{lemma}

\subsection{Proof of Theorem \ref{entropy}}

	{\bf First assertion}: Since $G_{K}\subseteq X$, we have $h_{top}^{UC}(T,G_{K})\leq h_{top}(T,X).$ Next, we will construct a subset (i.e. $Y_2$ in (\ref{equation-DC})) in $G_{K}$, such that $h_{top}^{UC}(T,Y_2)$ is close to $h_{top}(T,X)$. Such construction is the cornerstone of the proof.
	
	To begin with, since $(X,T)$ satisfies almost product property, we can let $g:\mathbb{N}\rightarrow\mathbb{N}$ and $m:\mathbb{R}^{+}\rightarrow \mathbb{N}$ be the two maps in the definition. Meanwhile, by using Proposition \ref{measure-center} for any non-empty open set $U$ with $U\cap C_T(X)\neq\emptyset$, we can fix an $\tilde{\epsilon}>0$, a point $z_{0}\in AP$ 
	and $L_{0}\in\mathbb{N}$ such that $\overline{B(z_{0},\tilde{\epsilon})}\subseteq U,$ and for any $l\geq 1$, there is $p\in \left[ l,l+L_{0}\right] $ such that $f^{p}(z_{0})\in B(z_{0},\tilde{\epsilon}/2).$ This implies that
	\begin{equation}\label{equation-AJ}
		\frac{\sharp\{0\leq p \leq lL_{0}:d(T^{p}(z_{0}),z_{0}) \leq \tilde{\epsilon}/2\}}{lL_{0}}\geq \frac{1}{L_{0}}.
	\end{equation}
	Take $l_{0}$ large enough such that
	\begin{equation}\label{AL}
		l_{0}L_{0}\geq m(\tilde{\epsilon}/2),\  \frac{g(l_{0}L_{0})}{l_{0}L_{0}}<\frac{1}{4L_{0}}.
	\end{equation}

	By the variational principle, for any $\eta>0$, there exists $\mu\in \cM^{e}(X,T)$ such that $h(T,\mu)>h_{top}(T,X)-\eta$, then by Proposition \ref{pro2.1}, there exists $\delta^{*}>0$ and $\epsilon^{*}>0$ so that for each neighbourhood $F$ of $\mu$ in $\cM(X,T)$, there exists $n^{*}_{F,\mu,\eta}$ such that for any $\mathcal{N} \geq n^{*}_{F,\mu,\eta}$, there exists a $(\delta^{*},\mathcal{N},3\epsilon^{*})$-separated set
	\begin{equation}\label{equation-DF}
	\Gamma_{\mathcal{N}}=\{z_{1},z_{2},\dots,z_{r}\} \subseteq X_{\mathcal{N},F}\ \text{such that}\ \sharp\Gamma_{\mathcal{N}} \geq e^{\mathcal{N}(h(T,\mu)-\eta)}.
	\end{equation}
	Take $\mathcal{N}$ large enough such that
	\begin{equation}\label{gn}
	\frac{g(\mathcal{N})}{\mathcal{N}}< \frac{\delta^{*}}{3},\ \mathcal{N}>m(\epsilon^{*}),\  \frac{1}{l_0L_0}e^{\mathcal{N}(h(T,\mu)-\eta)}>e^{(l_{0}L_{0}+\mathcal{N})(h(T,\mu)-2\eta)}.
	\end{equation}
	
	Let $\{\zeta_{k}\}$ and $\{\epsilon_{k}\}$ be two strictly decreasing sequences so that $\lim_{k\rightarrow \infty}\zeta_{k}=\lim_{k\rightarrow \infty}\epsilon_{k}=0$.
	Again,
	note that by Proposition \ref{measure-center} the almost periodic set AP is dense in $C_T(X)$. Thus, for any fixed $k$ there is a finite set $\Delta_{k}:=\{x_{1}^{k},x_{2}^{k},\dots,x_{t_{k}}^{k}\}\subseteq AP$ 
	and $L_{k}\in\mathbb{N}$ such that $\Delta_{k}$ is $\epsilon_{k}$-dense in $C_T(X),$ and for any $1\leq i\leq t_{k}$, any $l\geq 1$, there is $p_{i}\in \left[ l,l+L_{k}\right] $ such that $T^{p_{i}}x_{i}^{k}\in B(x_{i}^{k},\epsilon_{k})$. This implies that any $1\leq i\leq t_{k}$,
	\begin{equation}\label{LK}
	\frac{\sharp\{0\leq p_{i} \leq lL_{k}:d(T^{p_{i}}x_{i}^{k},x_{i}^{k})<\epsilon_{k}\}}{lL_{k}}\geq \frac{1}{L_{k}}.
	\end{equation}
	Take $l_{k}$ large enough such that
	\begin{equation}\label{lk}
	l_{k}L_{k}\geq m(\epsilon_{k}),\ \frac{g(l_{k}L_{k})}{l_{k}L_{k}}<\frac{1}{4L_{k}}.
	\end{equation}
	Because this is an inductive process, we may assume that the three sequences of $\{t_{k}\}$, $\{l_{k}\}$, $\{L_{k}\}$ are strictly increasing. 
	
	For any nonempty compact connected set $K\subseteq \cM(X,T)$, it follows from \cite[Theorem 5.1]{PS}\footnote{In fact, \cite[Theorem 5.1]{PS} states that if $(X,T)$ satisfies the almost product property and uniform separation property, then $G_{K}\neq\emptyset$ and $h_{top}^{B}(T,G_{K})=\inf\{h(T,\mu):\mu\in K\}$. However, by checking the proof, uniform separation property is only used to calculate entropy, so the hypotheses of Theorem \ref{entropy} are sufficient to guarantee $G_{K}\neq \varnothing$.} that $G_{K}\neq \varnothing$.
	Moreover, by Lemma \ref{AA}, there exists a sequence $\{\alpha_{1},\alpha_{2},\dots\}$ in $K$ such that
	\begin{equation}
	\overline{\{\alpha_{j}:j\in \mathbb{N}^{+},j>n\}}=K,\ \forall n \in \mathbb{N}^{+} \ \mathrm{and} \ \lim_{j\rightarrow \infty}d(\alpha_{j},\alpha_{j+1})=0.
	\end{equation}
	
	Let $x_{0}\in G_{K}$, there exists $n_{k}\in\mathbb{N}$ such that $P_{n_{k}}(x_{0})\in \cB(\alpha_{k},\zeta_{k})$. We can suppose that $n_{k}$ is large enough such that
	\begin{equation}\label{n}
	n_{k}>m(\epsilon_{k}),~~~\frac{t_{k}l_{k}L_{k}}{n_{k}}\leq \zeta_{k} ~~ \mathrm{and}~~\frac{g(n_{k})}{n_{k}}\leq \epsilon_{k}.
	\end{equation}
	We then choose a strictly increasing $\{N_{k}\}$, with $N_{k}\in\mathbb{N}$, such that
	\begin{equation}\label{nk}
	n_{k+1}+t_{k+1}l_{k+1}L_{k+1}\leq \zeta_{k}\left(\sum_{j=1}^{k}(n_{j}N_{j}+t_{j}l_{j}L_{j})+\mathcal{N}+l_{0}L_{0}\right)
	\end{equation}
	and
	\begin{equation}\label{AD}
	\sum_{j=1}^{k-1}(n_{j}N_{j}+t_{j}l_{j}L_{j})+\mathcal{N}+l_{0}L_{0}\leq \zeta_{k}\left(\sum_{j=1}^{k}(n_{j}N_{j}+t_{j}l_{j}L_{j})+\mathcal{N}+l_{0}L_{0}\right).
	\end{equation}
	
	Now we can define the sequences $\{n_{j}'\}$, $\{\epsilon_{j}'\}$, $\{\Gamma_{j}'\}$ inductively, by setting for:
	\begin{align*}
	&j = -1,\\
	&n_{-1}':=l_{0}L_{0}, \epsilon_{-1}':=\tilde{\epsilon}/2, \Gamma_{-1}':= \{z_0\} \ \mathrm{and}\ \mathrm{for} \\
	&j = 0,\\
	&n_{0}':=\mathcal{N}, \epsilon_{0}':=\epsilon^{*}, \Gamma_{0}':= \Gamma_{\mathcal{N}} \ \mathrm{and}\ \mathrm{for} \\
	&j = N_{1}+N_{2}+\dots+N_{k-1}+t_{1}+\dots+t_{k-1}+q \ \mathrm{with}\ 1\leq q \leq N_{k},\\
	&n_{j}':=n_{k}, \epsilon_{j}':=\epsilon_{k}, \Gamma_{j}':= \{x_{0}\} \ \mathrm{and}\ \mathrm{for} \\
	&j = N_{1}+N_{2}+\dots+N_{k}+t_{1}+\dots+t_{k-1}+q \ \mathrm{with}\ 1\leq q \leq t_{k},\\
	&n_{j}':=l_{k}L_{k}, \epsilon_{j}':=\epsilon_{k}, \Gamma_{j}':= \{x_{q}^{k}\}.
	\end{align*}
	For each $s \in \mathbb{N}^{+}$, define
	\begin{equation*}
	G_{s}^{(\mathcal{N})}:= \bigcap_{j=-1}^{s}\left(\bigcup_{x_{j}\in\Gamma_{j}'}T^{-M_{j-1}}B_{n_{j}'}(g;x_{j},\epsilon_{j}')\right),
	\end{equation*}
	where $M_{j}:=\sum_{l=-1}^{j}n_{l}'$ for any integer $j\geq -1$ and $M_{-2}:=0$.
	
	By almost product property, $G_{s}^{(\mathcal{N})}$ is a non-empty closed set. Let
    \begin{equation*}
    	G^{(\mathcal{N})}:=\bigcap_{s\geq 1}G_{s}^{(\mathcal{N})}.
    \end{equation*}
    By the nested structure of $G_{s}^N,$ one also has $G^{(\mathcal{N})}$ is non-empty and closed set.

Next we will prove the following:
	\begin{enumerate}
		\item $G^{(\mathcal{N})}\subseteq G_{K}$.
		\item $G^{(\mathcal{N})}\subseteq \left\{x\in X:C_T(X)\subset \omega_T(x)\right\}$.
		\item There exists $Y\subseteq G^{(\mathcal{N})}$ is a $(l_{0}L_{0}+\mathcal{N},\epsilon^{*})$-separated sets with $\sharp Y=\sharp\Gamma_{\mathcal{N}}$.
	\end{enumerate}

	\textbf{Proof of Item $(1)$}: Define the stretched sequence $\{\alpha_{l}'\}$ by
	\begin{equation}
	\alpha_{l}':=\alpha_{k} \quad \mathrm{if} \quad l_{0}L_{0}+\mathcal{N}+\sum_{j=1}^{k-1}(n_{j}N_{j}+t_{j}l_{j}L_{j})+1\leq l\leq l_{0}L_{0}+\mathcal{N}+\sum_{j=1}^{k}(n_{j}N_{j}+t_{j}l_{j}L_{j}).
	\end{equation}
	Then the sequence $\{\alpha_{l}'\}$ has the same limit-point set as the sequence of $\{\alpha_{k}\}$. If one has 
	\begin{equation}\label{equation-DJ}
	\lim_{l\rightarrow \infty}d(P_{l}(y),\alpha_{l}')=0\ \forall y\in G^{(\mathcal{N})},
	\end{equation}
	then the two sequences $\{P_{l}(y)\}$, $\{\alpha_{l}'\}$ have the same limit-point set. Thus, we can obtain $G^{N}\subseteq G_{K}$. Meanwhile, it follows from (\ref{nk}) that $\lim_{l\rightarrow \infty}\frac{M_{l+1}}{M_{l}}=1$. So from the definition of $\{\alpha_{l}'\}$, to verify (\ref{equation-DJ}) it is sufficient to prove that for any $y\in G^{(\mathcal{N})}$, one has
	\begin{equation}\label{equ:item1}
	\lim_{l\rightarrow \infty}d(P_{M_{l}}(y),\alpha_{M_{l}}')=0.
	\end{equation}
	To this end, noting that for any $l\in\mathbb{N}$, there exists a unique $k\in\mathbb{N}$ such that
$$
l_{0}L_{0}+\mathcal{N}+\sum_{j=1}^{k}(n_{j}N_{j}+t_{j}l_{j}L_{j})+1\leq M_{l}\leq l_{0}L_{0}+\mathcal{N}+\sum_{j=1}^{k+1}(n_{j}N_{j}+t_{j}l_{j}L_{j}).
$$

Denote $\mathcal{L}_k=l_{0}L_{0}+\mathcal{N}+\sum_{j=1}^{k}(n_{j}N_{j}+t_{j}l_{j}L_{j}).$ We split into two cases to discuss.

Case (1): If $M_{l}\leq \mathcal{L}_k+n_{k+1}N_{k+1}$, let $\alpha=\nu_{j}=\alpha_{k+1}$ for any $j$ in Lemma \ref{ds}, then by (\ref{n}), one has
	\begin{equation}\label{1}
	d(P_{M_{l}-\mathcal{L}_k}(T^{\mathcal{L}_k}y),\alpha_{k+1})
	\leq  \zeta_{k+1}+\epsilon_{k+1}+\frac{g(n_{k+1})}{n_{k+1}} \leq  \zeta_{k+1}+2\epsilon_{k+1}.
	\end{equation}

Case (2): Otherwise, if $M_{l}\geq \mathcal{L}_k+n_{k+1}N_{k+1}$, then one has
    \begin{align}\label{AC}
    	&d\left(P_{M_{l}-\mathcal{L}_k}\left(T^{\mathcal{L}_k}y\right),\alpha_{k+1}\right) \\ \leq~~&\frac{n_{k+1}N_{k+1}}{M_{l}-\mathcal{L}_k}d\left(P_{n_{k+1}N_{k+1}}\left(T^{\mathcal{L}_k}y\right),\alpha_{k+1}\right)+\frac{M_{l}-\mathcal{L}_k-n_{k+1}N_{k+1}}{M_{l}-\mathcal{L}_k}\times 1\notag\\
    	\leq~~& 1\times (\zeta_{k+1}+2\epsilon_{k+1})+\frac{t_{k+1}l_{k+1}L_{k+1}}{n_{k+1}N_{k+1}} &(\text{using}\ (\ref{1})) \notag\\
    	\leq~~& 2\zeta_{k+1}+2\epsilon_{k+1}. &(\text{using}\ (\ref{n}))\notag
    \end{align}
	Meanwhile, by using (\ref{1}), one also has
	\begin{equation}\label{AB}
        d\left(P_{n_{k}N_{k}}\left(T^{\mathcal{L}_{k-1}}y\right),\alpha_{k+1}\right)\leq d\left(P_{n_{k}N_{k}}\left(T^{\mathcal{L}_{k-1}}y\right),\alpha_{k}\right)+d(\alpha_{k+1},\alpha_{k})
		\leq \zeta_{k}+2\epsilon_{k}+d(\alpha_{k},\alpha_{k+1}).
	\end{equation}
	Thus,
	\begin{align*}
	&d(P_{M_{l}}(y),\alpha_{M_{l}}')=d(P_{M_{l}}(y),\alpha_{k+1})\\
\leq~~&\frac{\mathcal{L}_{k-1}}{M_{l}}d(P_{\mathcal{L}_{k-1}}(y),\alpha_{k+1})+\frac{n_{k}N_{k}}{M_{l}}d(P_{n_{k}N_{k}}(T^{\mathcal{L}_{k-1}}y),\alpha_{k+1})\\
	  &+\frac{t_{k}l_{k}L_{k}}{M_{l}}d(P_{t_{k}l_{k}L_{k}}(T^{\mathcal{L}_{k-1}+n_{k}N_{k}}y),\alpha_{k+1})+\frac{M_{l}-\mathcal{L}_k}{M_{l}}d\left(P_{M_{l}-\mathcal{L}_k}(T^{\mathcal{L}_k}y),\alpha_{k+1}\right)\\
\leq~~& \frac{\mathcal{L}_{k-1}}{\mathcal{L}_k}\times 1 + 1\times (\zeta_{k}+2\epsilon_{k}+d(\alpha_{k},\alpha_{k+1})) + \frac{t_{k}l_{k}L_{k}}{n_{k}} + 2\zeta_{k+1}+2\epsilon_{k+1}\\
\leq~~& \zeta_{k}+\zeta_{k}+2\epsilon_{k}+d(\alpha_{k},\alpha_{k+1})+\zeta_{k}+ 2\zeta_{k+1}+2\epsilon_{k+1}\\
=~~	  &5\zeta_{k}+4\epsilon_{k}+d(\alpha_{k},\alpha_{k+1}).
	\end{align*}
	The second inequality is followed by using (\ref{1}), (\ref{AB}) and (\ref{AC}). The last inequality is followed by using (\ref{AD}) and (\ref{n}). Taking $k$ to infinity on both sides, and note that LHS goes to zero, this directly yields \eqref{equ:item1}, and the proof of item $(1)$ is completed.
	
	\textbf{Proof of Item $(2)$}: Fix $x\in G^{(\mathcal{N})}$. By construction of $G^{(\mathcal{N})}$, for any fixed $k\geq 1$, there is $a=a_{k}$ such that for any $j=1,\dots,t_{k}$, there is $\Lambda^{j}\subseteq \Lambda_{l_{k}L_{k}}$
	\begin{equation}
	\max\{d(T^{a+l+(j-1)l_{k}L_{k}}x,T^{l}x_{j}^{k}): l\in \Lambda^{j}\}\leq \epsilon_{k}.
	\end{equation}
	By (\ref{lk}), it follows that
	\begin{equation}
	\frac{\sharp\Lambda^{j}}{l_{k}L_{k}}\geq 1-\frac{g(l_{k}L_{k})}{l_{k}L_{k}}\geq 1-\frac{1}{4L_{k}}.
	\end{equation}
	Together with (\ref{LK}), we get that for any $j=1,\dots,t_{k}$ there is $p_{j}\in \left [ 0,l_{k}L_{k}-1\right ]$ such that
	\begin{equation}
	d(T^{a+p_{j}+(j-1)l_{k}L_{k}}x,T^{p_{j}}x_{j}^{k})\leq \epsilon_{k} \ \text{and} \ d(x_{j}^{k},T^{p_{j}}x_{j}^{k})\leq \epsilon_{k}.
	\end{equation}
	This implies $d(T^{a+p_{j}+(j-1)l_{k}L_{k}}x,x_{j}^{k})\leq 2\epsilon_{k}$, so that the orbit of $x$ is $3\epsilon_{k}$-dense in $C_T(X)$. By the
	arbitrariness of $k$, one has $C_T(X)\subset \omega_T(x)$. This completes the proof of item $(2)$.

	\textbf{Proof of Item $(3)$}: Recall $\Gamma_{N}=\{z_1,\cdots,z_{r}\}$ is a separated set given previously in (\ref{equation-DF}). We define
$$
G^{i}:= B_{l_0L_0}(g;z_0,\tilde{\epsilon})\cap T^{-l_0L_0}B_{\mathcal{N}}(g;z_{i},\epsilon^{*})\cap \left(\bigcap_{j\geq 1}\bigcup_{x_{j}\in\Gamma_{j}'}T^{-M_{j-1}}B_{n_{j}'}(g;x_{j},\epsilon_{j}')\right)\ \forall i=1,2,\cdots,r.
$$
By the almost product property, each $G^{i}$ is a non-empty and closed set. Actually, $\cup_{i=1}^{r} G^{i}$ is a rearrangement of $G^{(\mathcal{N})}.$ So
$$
G^{(\mathcal{N})}=\cup_{i=1}^{r} G^{i}.
$$
Taking any $y_{i}\in G^{i}$ for $i=1,\dots,r$, and define
$$
Y:=\{y_{1},\dots,y_{r}\}\subseteq \cup_{i=1}^{r} G^{i}=G^{(\mathcal{N})}.
$$
We claim that set $Y$ is a $(l_0L_0+\mathcal{N},\epsilon^{*})$-separated sets.
To this end, fix $y_{u},y_{v}\in Y$ with $y_{u}\neq y_{v}$. By the definition of $(\delta^{*},\mathcal{N},3\epsilon^{*})$-separated subset, one has
	\begin{equation*}
	\frac{\sharp\{j:d(T^{j}z_{u},T^{j}z_{v})>3\epsilon^{*},j\in\Lambda_{\mathcal{N}}\}}{\mathcal{N}} \geq \delta^{*}.
	\end{equation*}
	By (\ref{gn})
	\begin{equation*}
	\frac{\sharp\{j:d(T^{l_0L_0+j}y_{u},T^{j}z_{u})\leq\epsilon^{*},j\in\Lambda_{\mathcal{N}}\}}{\mathcal{N}} \geq 1-\frac{g(\mathcal{N})}{\mathcal{N}} \geq 1-\frac{\delta^{*}}{3},
	\end{equation*}
	\begin{equation*}
	\frac{\sharp\{j:d(T^{l_0L_0+j}y_{v},T^{j}z_{v})\leq\epsilon^{*},j\in\Lambda_{\mathcal{N}}\}}{\mathcal{N}} \geq 1-\frac{g(\mathcal{N})}{\mathcal{N}} \geq 1-\frac{\delta^{*}}{3}.
	\end{equation*}
	Thus there exists $j_{uv}\in \Lambda_{\mathcal{N}}$ such that
	\begin{equation}
	d(T^{j_{uv}}z_{u},T^{j_{uv}}z_{v})>3\epsilon^{*},\
	d(T^{l_0L_0+j_{uv}}y_{u},T^{j_{uv}}z_{u})\leq\epsilon^{*},\
	d(T^{l_0L_0+j_{uv}}y_{v},T^{j_{uv}}z_{v})\leq\epsilon^{*},
	\end{equation}
	which implies $$d(T^{l_0L_0+j_{uv}}y_{u},T^{l_0L_0+j_{uv}}y_{v})>\epsilon^{*}.$$ Consequently, we get that $Y:=\{y_{1},\dots,y_{r}\} $ is a $(l_0L_0+\mathcal{N},\epsilon^{*})$-separated sets for $G^{(\mathcal{N})}$ with $\sharp Y=\sharp\Gamma_{\mathcal{N}}\geq e^{\mathcal{N}(h(T,\mu)-\eta)}$. This completes the proof of Item $(3)$.\qed\\

Finally we will use item (3) to see that $G^{(\mathcal{N})}$ carries the full upper capacity entropy.
By construction of $G^{(\mathcal{N})}$ and (\ref{AL}), for each $k\in \{1,2,\dots,r\}$, there is $\Lambda^{k}\subseteq \Lambda_{l_{0}L_{0}}$ such that
\begin{equation}
	\max\{d(T^{l}y_r,T^{l}z_{0}):l\in \Lambda^{k}\}\leq \tilde{\epsilon}/2,\
	\frac{\sharp\Lambda^{k}}{l_{0}L_{0}}\geq 1-\frac{g(l_{0}L_{0})}{l_{0}L_{0}}\geq 1-\frac{1}{4L_{0}}.
\end{equation}
Together with (\ref{equation-AJ}) we get that there is $q_{k}\in \left [ 0,l_{0}L_{0}-1\right ]$ such that
\begin{equation}
	d(T^{q_{k}}y_k,T^{q_{k}}z_{0})\leq \tilde{\epsilon}/2 \ \mathrm{and} \ d(z_{0},T^{q_k}z_{0})\leq \tilde{\epsilon}/2,
\end{equation}
which implies $d(T^{q_{k}}y_r,z_{0})\leq \tilde{\epsilon}$.

Using the pigeon-hole principle, we obtain that there is a subset $Y_1\subseteq Y$ and $q \in \left [ 0,l_{0}L_{0}-1\right ]$ such that
\begin{equation*}
	d(T^{q}y,z_{0})\leq \tilde{\epsilon}\ \text{for any}\ y \in Y_1\ \text{and}\ \sharp Y_1\geq\frac{\sharp Y}{l_0L_0} \geq \frac{1}{l_0L_0}e^{\mathcal{N}(h(T,\mu)-\eta)}.
\end{equation*}  
Therefore, we can define
\begin{equation}\label{equation-DC}
	Y_2:= \{f^{q}y:y \in Y_{1}\},
\end{equation} 
and it is easy to see that $Y_2$ is also a $(l_0L_0+\mathcal{N},\epsilon^{*})$-separated sets. Moreover, by (\ref{gn}), $$\sharp Y_2= \sharp Y_1\geq \frac{1}{l_0L_0}e^{\mathcal{N}(h(T,\mu)-\eta)}\geq e^{(l_{0}L_{0}+\mathcal{N})(h(T,\mu)-2\eta)}.$$ Meanwhile, since $G_{K}$ is $T$-invariant, one also has $Y_2\subset G_{K}\cap U\cap \{x\in X:C_T(X)\subset \omega_T(x)\}.$ 
Thus
\begin{equation*}
	\begin{split}
		&h_{top}^{UC}(T,G_{K}\cap U\cap \{x\in X:C_T(X)\subset \omega_T(x)\})\\
		\geq &\limsup_{\mathcal{N} \rightarrow \infty}\frac{1}{l_0L_0+\mathcal{N}}\log s_{l_0L_0+\mathcal{N}}(Z,\epsilon^{*})\\
		\geq &h(T,\mu)-2\eta > h_{top}(T,X)-3\eta.
	\end{split}
\end{equation*} 
By the arbitrariness of $\eta$, this directly means $$h_{top}^{UC}(T,G_{K}\cap U\cap \{x\in X:C_T(X)\subset \omega_T(x)\})\geq h_{top}(T,X),$$ which completes the proof of the first assertion.

\medskip
	
	{\bf Second assertion}: If there is an invariant measure with full support, then by Proposition \ref{measure-center}, one has $C_T(X)=X.$ So $\{x\in X:C_T(X)\subset \omega_T(x)\}\subset Tran.$ Combining with the first assertion, we completes the proof of the second assertion.


\section{Packing entropy formula}\label{sec:proofentropy2}
In this section, we prove Theorem \ref{packing-entropy}. Before proving Theorem \ref{packing-entropy}, we need some basic results which is also useful for the
proof of our main theorems. 
\subsection{Some lemmas}
Brin and Katok \cite{BK} introduced the \emph{local measure-theoretical upper entropies of $\mu$} for each $\mu\in\mathcal{M}(X)$, by
\begin{equation}
	\overline{h}_{\mu}(T)=\int\overline{h}_{\mu}(T,x)d\mu,
\end{equation}
where
\begin{equation}\label{equation-GI}
	\overline{h}_{\mu}(T,x)=\lim_{\epsilon\to 0}\limsup_{n\to\infty}-\frac{1}{n}\log\mu(B_{n}(x,\epsilon)).
\end{equation}
Deducing from the proof of Theorem 1.3 of \cite{Feng-Huang}, we have the following result.
\begin{lemma}\label{lemma-aa}
	If $Z\subset X$ is an Borel set, then
	\begin{equation}
		h^{P}_{top}(T,Z)\geq \sup\{\overline{h}_{\mu}(T):\mu\in\mathcal{M}(X),\ \mu(Z)=1\}.
	\end{equation}
\end{lemma}
Next, we recall the entropy-dense property.
\begin{definition}[Entropy-dense property]
	We say $(X,T)$ satisfies the \emph{entropy-dense property} if for any $\mu \in \mathcal{M}(X,T)$, for any neighbourhood $G$ of $\mu$ in $\mathcal{M}(X)$, and for any $\eta >0$, there exists a closed $T$-invariant set $\Lambda_{\mu}\subseteq X$, such that  $\mathcal{M}(T,\Lambda_{\mu})\subseteq G$ and $h_{top}(T,\Lambda_{\mu})>h(T,\mu)-\eta$. By the classical variational principle, it is equivalent that for any neighbourhood $G$ of $\mu$ in $\mathcal{M}(X)$, and for any $\eta >0$, there exists a $\nu \in \mathcal{M}^{e}(X,T)$ such that $h(T,\nu)>h(T,\mu)-\eta$ and $\mathcal{M}(T,S_{\nu})\subseteq G$.
\end{definition}
\begin{lemma}\cite[Proposition 2.3 (1)]{PS}\label{entropy-dense}
	Suppose that $(X,T)$ has almost product property, then the entropy-dense property holds.
\end{lemma}

\subsection{Proof of Theorem \ref{packing-entropy}}
{\bf First assertion:} By checking the Part II \cite[Page 398]{ZCC2012} and \cite[Corollary 3.5]{ZCC2012}, we get that for any TDS $(X,T)$ without additional hypothesis, $h_{top}^{P}(T,G_{K})\leq\sup\{h(T,\mu):\mu\in K\}$ holds for any nonempty compact convex set $K\subseteq \mathcal{M}(X,T)$. To show the reverse inequality, we will construct a subset (see $G$ in (\ref{equation-DG})) of $G_{K}$ such that $h^{P}_{top}(T,G)$ is close to $\sup\{h(T,\mu):\mu\in K\}.$ This construction is the cornerstone of the proof.

To begin with, since $(X,T)$ satisfies almost product property, we can let $g:\mathbb{N}\rightarrow\mathbb{N}$ and $m:\mathbb{R}^{+}\rightarrow \mathbb{N}$ be the two maps in the definition. Meanwhile, by using Proposition \ref{measure-center}
for any non-empty open set $U$ with $U\cap C_T(X)\neq\emptyset$, we can fix an $\tilde{\epsilon}>0$, a point $z_{0}\in AP$ 
and $L_{0}\in\mathbb{N}$ such that $\overline{B(z_{0},\tilde{\epsilon})}\subseteq U,$ and for any $l\geq 1$, there is $p\in \left[ l,l+L_{0}\right] $ such that $T^{p}(z_{0})\in B(z_{0},\tilde{\epsilon}/2).$ This implies that
\begin{equation}\label{equation-AP}
	\frac{\sharp\{0\leq p \leq lL_{0}:d(T^{p}(z_{0}),z_{0}) \leq \tilde{\epsilon}/2\}}{lL_{0}}\geq \frac{1}{L_{0}}.
\end{equation}
Take $l_{0}$ large enough such that
\begin{equation}\label{AQ}
	l_{0}L_{0}\geq m(\tilde{\epsilon}/2),\  \frac{g(l_{0}L_{0})}{l_{0}L_{0}}<\frac{1}{4L_{0}}.
\end{equation}

For any $\eta>0,$ let $$H^{*}:= \sup\{h(T,\mu):\mu\in K\}-\eta.$$
For the $\eta$, there exists $\mu_{\eta}\in K$ such that
	\begin{equation}\label{equation-ab}
	h(T,\mu_{\eta})\geq H^{*}.
	\end{equation}
	
	Since $K$ is convex and compact, we can choose $\{\mu_{n,i}:n\geq 1, 1\leq i\leq \mathfrak{M}_{n}\}\subseteq K$ such that the following four properties hold,
	\begin{align}
	K\subset \bigcup_{i=1}^{\mathfrak{M}_{n}}B(\mu_{n,i},\frac{1}{n})\ \forall n\in\mathbb{N},\\
	d(\mu_{n,i},\mu_{n,i+1})\leq \frac{1}{n}\ \forall n\in\mathbb{N},1\leq i< \mathfrak{M}_{n},\\
	d(\mu_{n,\mathfrak{M}_{n}},\mu_{n+1,1})\leq \frac{1}{n}\ \forall n\in\mathbb{N},\\
	\mu_{n,\mathfrak{M}_{n}}=\mu_{\eta}~~ \forall n\in\mathbb{N}.
	\end{align}
	Based on these properties, we can define
	$$
	\{\alpha_{k}\}_{k}:=\{\mu_{1,1},\mu_{1,2},\dots,\mu_{1,\mathfrak{M}_{1}},\mu_{2,1},\mu_{2,2},\dots\},
	$$
	and has
	\begin{equation}
	\overline{\{\alpha_{k}:k\in \mathbb{N}^{+},k>n\}}=K, \forall n \in \mathbb{N}^+,~~ \mathrm{and} ~~ \lim_{k\rightarrow \infty}d(\alpha_{k},\alpha_{k+1})=0,
	\end{equation}
	and $\{\alpha_{k}\}_{k}$ satisfies 
	\begin{equation}
	\alpha_{\mathfrak{N}_{n}}=\mu_{\eta}, \forall n\in\mathbb{N},
	\end{equation}
    where $\mathfrak{N}_{n}=\sum_{i=1}^{n}\mathfrak{M}_{i}$.
	
	On the other hand, by \cite[Lemma 6.2]{PS} we can find $\epsilon^{*}>0, \delta^{*}>0,$ and for any $j\in\mathbb{N},$ a finite convex combination of ergodic measures with rational coefficients
	$$
	\beta_{j}:=\sum_{i=1}^{p_{j}} a_{i, j} \mu_{i, j},
	$$
	so that $\left\{\beta_{j}\right\}$ converges to $\mu_{\eta},$ and such that
	\begin{equation}\label{equation-AF}
		\sum_{i=1}^{p_{j}} a_{i, j} \underline{s}\left(\mu_{i, j}, \delta^{*}, \epsilon^{*}\right)>h(T,\mu_{\eta})-\eta,
	\end{equation}
	where $\underline{s}(\mu; \delta, \epsilon):=\inf _{F \ni \mu} \liminf _{n \rightarrow \infty} \frac{1}{n} \log N(F;\delta, n, \epsilon),$ the infimum is taken over any base of neighborhoods of $\nu,$ and 
	\begin{equation}\label{equation-DH}
		N(F;\delta,n,\epsilon):=\mbox{maximal cardinality of a}~(\delta,n,\epsilon)\mbox{-separated subset of}~X_{n,F}.
	\end{equation}
	 
	Let $\{\zeta_{k}\}$ and $\{\epsilon_{k}\}$ be two strictly decreasing sequences so that $\lim_{k\rightarrow \infty}\zeta_{k}=\lim_{k\rightarrow \infty}\epsilon_{k}=0$ with $\zeta_{1}<\eta$ and  $\epsilon_{1}<\frac{1}{4}\epsilon^{*}$. We assume that
	$$
	d\left(\mu_{\eta}, \beta_{j}\right) \leq \zeta_{\mathfrak{N}_{j}}.
	$$
	Note that by Proposition \ref{measure-center} the almost periodic set AP is dense in $C_T(X)$. Thus, for any fixed $k$ there is a finite set $\Delta_{k}:=\{x_{1}^{k},x_{2}^{k},\dots,x_{t_{k}}^{k}\}\subseteq AP$ 
    and $L_{k}\in\mathbb{N}$ such that $\Delta_{k}$ is $\epsilon_{k}$-dense in $C_T(X),$ and for any $1\leq i\leq t_{k}$, any $l\geq 1$, there is $p_{i}\in \left[ l,l+L_{k}\right] $ such that $T^{p_{i}}x_{i}^{k}\in B(x_{i}^{k},\epsilon_{k})$. This implies that any $1\leq i\leq t_{k}$,
	\begin{equation}
	\frac{\sharp\{0\leq p_{i} \leq lL_{k}:d(T^{p_{i}}x_{i}^{k},x_{i}^{k})<\epsilon_{k}\}}{lL_{k}}\geq \frac{1}{L_{k}}.
	\end{equation}
	Take $l_{k}$ large enough such that
	\begin{equation}\label{equ-lk}
	l_{k}L_{k}\geq m(\epsilon_{k}), \frac{g(l_{k}L_{k})}{l_{k}L_{k}}<\frac{1}{4L_{k}}.
	\end{equation}
	Because this is an inductive process, we may assume that the three sequences of $\{t_{k}\}$, $\{l_{k}\}$, $\{L_{k}\}$ are strictly increasing.
	
	If $k=\mathfrak{N}_{j}$ for some $j\in\mathbb{N}$, there exists an integer $n_{k}$ so each element of $\left\{n_{k} a_{i, j}\right\}_{i=1}^{p_j}$ is an integer,
	\begin{equation}\label{equation-AG}
		N\left(\cB\left(\mu_{i, j}, \zeta_{k}\right) ; \delta^{*}, a_{i, j} n_{k}, \epsilon^{*}\right) \geq e^{a_{i, j} n_{k}\left(\underline{s}\left(\mu_{i, j}, \delta^{*}, \epsilon^{*}\right)-\eta\right)},\ \forall i=1, \ldots, p_{j},
	\end{equation}
	and moreover
	\begin{equation}\label{equ-AE}
		a_{i, j}n_{k}>m(\epsilon_{k}),\ \frac{t_{k}l_{k}L_{k}}{n_{k}}\leq \zeta_{k},\ \delta^{*}a_{i, j} n_{k} > 2g(a_{i, j} n_{k}) + 1 \ \mathrm{and} \ \frac{g(a_{i, j}n_{k})}{a_{i, j}n_{k}}\leq \epsilon_{k}\ \forall i=1, \ldots, p_{j}.
	\end{equation}
    Based on (\ref{equation-AG}), let $\Gamma_{i, j}$ be a $\left(\delta^{*}, a_{i, j} n_{k}, \varepsilon^{*}\right)$-separated subset of $X_{a_{i, j} n_{k}, \cB\left(\mu_{i, j}, \zeta_{k}\right)}$ such that the cardinality of $\Gamma_{i, j}$ is at least
    $$
    e^{a_{i, j} n_{k}\left(\underline{s}\left(\mu_{i, j}, \delta^{*}, \varepsilon^{*}\right)-\eta\right)},
    $$
    and define
    $$
    \Gamma_{j}:=\prod_{i=1}^{p_{j}} \Gamma_{i, j}.
    $$
    By using (\ref{equation-AF}) and (\ref{equation-AG}),
    \begin{equation}\label{equation-AH}
    	\sharp\Gamma_{j} \geq e^{n_{k} (h(T,\mu_{\eta})-2\eta)}.
    \end{equation}
    We also denote the elements of $\Gamma_{j}$ by
    $$
    \mathbf{x}_{j}:=\left(x_{1, j}, \ldots, x_{p_{j}, j}\right) \ \text {with}\ x_{i,j}\in \Gamma_{i, j},
    $$
    and set
    $$
    B_{n_{k}}\left(g; \mathbf{x}_{j}, \epsilon_{k}\right):=\bigcap_{i=1}^{p_{j}} T^{-\left(a_{1, j}+\cdots+a_{i-1, j}\right) n_{k}} B_{a_{i, j} n_{k}}\left(g; x_{i, j}, \epsilon_{k}\right) \ \text { with } a_{0,j}:=0.
    $$
	By Lemma \ref{entropy-dense}, for any $k\in \mathbb{N}\setminus\{\mathfrak{N}_{j}:j\in\mathbb{N}\},$ there exists $x_k\in X$ and $\tilde{n}_k\in\mathbb{N}$ such that $d(P_{n}(x_k),\alpha_{k})<\zeta_{k}$ holds for any $n\geq \tilde{n}_k.$ Take $n_{k}$ large enough such that
	\begin{equation}\label{equ-AB}
	n_k\geq \tilde{n}_k,\ n_{k}>m(\epsilon_{k}),\ \frac{t_{k}l_{k}L_{k}}{n_{k}}\leq \zeta_{k}\ \mathrm{and} \ \frac{g(n_{k})}{n_{k}}\leq \epsilon_{k}.
	\end{equation}
	We can then choose a strictly increasing $\{N_{k}\}$, with $N_{k}\in\mathbb{N}$, such that
	\begin{equation}\label{nkl}
	n_{k+1}+t_{k+1}l_{k+1}L_{k+1}\leq \zeta_{k}\left(l_{0}L_{0}+\sum_{j=1}^{k}\left(n_{j}N_{j}+t_{j}l_{j}L_{j}\right)\right),
	\end{equation}
	\begin{equation}\label{equ-AC}
	l_{0}L_{0}+\sum_{j=1}^{k-1}(n_{j}N_{j}+t_{j}l_{j}L_{j})\leq \zeta_{k}\left(l_{0}L_{0}+\sum_{j=1}^{k}(n_{j}N_{j}+t_{j}l_{j}L_{j})\right),
	\end{equation}
    and
    \begin{equation}\label{equation-ac}
    (H^{*}-2\eta-\zeta_{\mathfrak{N}_{i}})\left(l_{0}L_{0}+\sum_{j=1}^{\mathfrak{N}_{i}}(n_{j}N_{j}+t_{j}l_{j}L_{j})\right)\leq (H^{*}-2\eta)n_{\mathfrak{N}_{i}}N_{\mathfrak{N}_{i}}\ \forall i.
    \end{equation}

	Now we can define the sequences $\{n_{j}'\}$, $\{\epsilon_{j}'\}$, $\{\Gamma_{j}'\}$ inductively, by setting for:
	\begin{align*}
	&j = 0,\\
	&n_{0}':=l_0L_0, \epsilon_{0}':=\tilde{\epsilon}/2, \Gamma_{0}':= \{z_0\} \ \mathrm{and}\ \mathrm{for} \\
	&j = N_{1}+N_{2}+\dots+N_{k-1}+t_{1}+\dots+t_{k-1}+q \ \mathrm{with}\ k\in \mathbb{N}\setminus\{\mathfrak{N}_{i}:i\in\mathbb{N}\}\ \mathrm{and}\ 1\leq q \leq N_{k}\\
	&n_{j}':=n_{k}, \epsilon_{j}':=\epsilon_{k},  \Gamma_{j}':= \{x_k\} \ \mathrm{and}\ \mathrm{for} \\
	&j = N_{1}+N_{2}+\dots+N_{k}+t_{1}+\dots+t_{k-1}+q \ \mathrm{with}\ 1\leq q \leq t_{k},\\
	&n_{j}':=l_{k}L_{k}, \epsilon_{j}':=\epsilon_{k}, \Gamma_{j}':= \{x_{q}^{k}\}\ \mathrm{and}\ \mathrm{for} \\
	&j = N_{1}+N_{2}+\dots+N_{k-1}+t_{1}+\dots+t_{k-1}+q\ \mathrm{with}\ k=\mathfrak{N}_{i}\ \mathrm{for}\ \mathrm{some}\ i\in\mathbb{N}\ \mathrm{and}\ 1\leq q \leq N_{k}\\
	&n_{j}':=n_{k}, \epsilon_{j}':=\epsilon_{k},  \Gamma_{j}':= \Gamma_{i}.
	\end{align*}
	For each $s \in \mathbb{N}^{+}$, define
	\begin{equation}\label{equ-AD}
	G_{s}^{\eta}:= \bigcap_{j=0}^{s}\left(\bigcup_{x_{j}\in\Gamma_{j}'}T^{-M_{i-1}}B_{n_{j}'}(g;x_{j},\epsilon_{j}')\right),
	\end{equation}
	where $M_{j}:=\sum_{l=0}^{j}n_{l}'$ for any $j\in\mathbb{N}$ and $M_{-1}=0$.

	By almost product property, $G_{s}^{\eta}$ is a non-empty closed set.
	We can label each set which is obtained by developing (\ref{equ-AD}) by the branches of a labelled tree of height $s$. To be more specific, a branch is labelled by $(x_{1},\dots,x_{s})$ with
	$x_{j}\in\Gamma_{j}'$. By \cite[Lemma 3.8(1)]{HTW}, let $x_{j},y_{j}\in\Gamma_{j}'$ with $x_{j}\neq y_{j}$, if $x\in B_{n_{j}'}(g;x_{j},\epsilon_{j}')$ and $y\in B_{n_{j}'}(g;y_{j},\epsilon_{j}')$, then
	$d_{n_{j}}(x,y)>2\epsilon$ where $\epsilon=\frac{1}{4}\epsilon^{*}$. So $G_{s}^{\eta}$ is a closed set that is the disjoint union of non-empty closed sets $G_{s}^{\eta}(x_{1},\dots,x_{s})$ labelled
	by $(x_{1},\dots,x_{s})$ with $x_{j}\in\Gamma_{j}'$. Two different sequences label two different sets. We choose $x_{\xi}\in G^{\eta}_{k}(\xi)$ for any $\xi=(x_{1},\dots,x_{s})\in \Gamma_{1}'\times\dots\times\Gamma_{s}'$, then $F_{s}:=\{x_{\xi}:\xi\in\Gamma_{1}'\times\dots\times\Gamma_{s}' \}$ is a $(M_{s},2\epsilon)$-separated set.
	
	Let 
	\begin{equation*}
		G^{\eta}:=\bigcap_{s\geq 1}G_{s}^{\eta},
	\end{equation*}
By the nested structure of $G_{s}^{\eta},$ one has $G^{\eta}$ is a closed set and $G^{\eta}$ is the disjoint union of non-empty closed sets $G^{\eta}(x_{1},x_{2},\dots)$ labelled
by $(x_{1},x_{2},\dots)$ with $x_{j}\in\Gamma_{j}'$. Two different sequences label two different sets. Using the same method in the proof of Items $(1)$ and $(2)$ of Theorem \ref{entropy}, we have $G^{\eta}\subseteq G_{K}\cap \{x\in X:C_T(X)\subset \omega_T(x)\}$.

Next we will prove
\begin{equation}\label{equ:packinghardpart}
h_{top}^{P}(T,G^{\eta})\geq H^{*}-2\eta.
\end{equation}
To this end, define $\mu_{k}:=\frac{1}{\sharp\Gamma_{1}'\dots\sharp\Gamma_{k}'}\sum_{x\in F_{k}}\delta_{x}$. Suppose $\mu=\lim_{n\to\infty}\mu_{k_{n}}$ for some $k_{n}\to \infty$. For any fixed $l$ and all $p\geq 0$, since $\mu_{l+p}(G^{\eta}_{l+p})=1$ and $G^{\eta}_{l+p}\subseteq G^{\eta}_{l},$ one has $\mu_{l+p}(G^{\eta}_{l})=1.$
Then $\mu(G^{\eta}_{l})\geq \limsup_{n\to\infty}\mu_{k_{n}}(G^{\eta}_{l})=1$. It follows that
	$\mu(G^{\eta})=\lim_{n\to \infty}\mu(G^{\eta}_{l})=1$.
	
	Note that for any $x\in G^{\eta}(x_{1},\dots,x_{s},\dots)$, $y\in G^{\eta}_{s}(y_{1},\dots,y_{k})$, if $(x_{1},\dots,x_{s})\neq (y_{1},\dots,y_{s})$ for some $1\leq s\leq k$, one has
	$y\not\in B_{M_{s}}(x,2\varepsilon)$. For any $i$, by letting $\mathfrak{L}_{i}=\sum_{j=1}^{\mathfrak{N}_{i}-1}(N_{j}+t_{j})+N_{\mathfrak{N}_{i}},$ then for any $k_n>\mathfrak{L}_{i}$ and any $x\in G^{\eta}$ one has
	\begin{align*}
		\mu_{k_{n}}(B_{M_{\mathfrak{L}_{i}}}(x,\epsilon))\leq& \frac{1}{\sharp\Gamma_{1}'\sharp\Gamma_{2}'\dots\sharp\Gamma_{\mathfrak{L}_{i}}'}\\
		\leq& \frac{1}{\sharp\Gamma_{\mathfrak{L}_{i}-N_{\mathfrak{N}_{i}}+1}'\sharp\Gamma_{\mathfrak{L}_{i}-N_{\mathfrak{N}_{i}}+2}'\dots\sharp\Gamma_{\mathfrak{L}_{i}}'}\\
		\leq& e^{-n_{\mathfrak{N}_{i}}N_{\mathfrak{N}_{i}}(H^*-2\eta)} &(\text{using} (\ref{equation-AH})\ \text{and}\ (\ref{equation-ab}))\\
		\leq& e^{-M_{\mathfrak{L}_{i}}\left(H^{*}-2\eta-\zeta_{\mathfrak{N}_{i}}\right)}. &(\text{using} (\ref{equation-ac}))
	\end{align*}
    Then
    \begin{equation}
    \mu(B_{M_{\mathfrak{L}_{i}}}(x,\epsilon))\leq\liminf_{n\rightarrow \infty}\mu_{k_{n}}(B_{M_{\mathfrak{L}_{i}}}(x,\epsilon))
    \leq e^{-M_{\mathfrak{L}_{i}}\left(H^{*}-2\eta-\zeta_{\mathfrak{N}_{i}}\right)}.
    \end{equation}
    Using (\ref{equation-GI}), we have
    \begin{equation}
    \begin{split}
    \overline{h}_{\mu}(T,x)
    &\geq \limsup_{i\to\infty}-\frac{1}{M_{\mathfrak{L}_{i}}}\log\mu(B_{M_{\mathfrak{L}_{i}}}(x,\epsilon)) \\
    &\geq \limsup_{i\to\infty}H^{*}-2\eta-\zeta_{\mathfrak{N}_{i}}\\
    &= H^{*}-2\eta.
    \end{split}
    \end{equation}
    Taking supermum over all $\mu \in \mathcal{M}(X)$ with $\mu(G^{\eta})=1,$ and applying Lemma \ref{lemma-aa}, we eventually have
    $$h^{P}_{top}(T,G^{\eta})\geq H^{*}-2\eta= \sup\{h(T,\mu):\mu\in K\}-3\eta.$$
    
    By construction of $ G^{\eta}$ and (\ref{AQ}), for any $x\in G^{\eta}$, there is $\Lambda^{x}\subseteq \Lambda_{l_{0}L_{0}}$ such that
    \begin{equation}
    	\max\{d(T^{l}x,T^{l}z_{0}):l\in \Lambda^{x}\}\leq \tilde{\epsilon}/2,\ 
    	\frac{\sharp\Lambda^{x}}{l_{0}L_{0}}\geq 1-\frac{g(l_{0}L_{0})}{l_{0}L_{0}}\geq 1-\frac{1}{4L_{0}}.
    \end{equation}
    Together with (\ref{equation-AP}), there is $q_{x}\in \left [ 0,l_{0}L_{0}-1\right ]$ such that
    \begin{equation}
    	d(T^{q_{x}}x,T^{q_{k}}z_{0})\leq \tilde{\epsilon}/2 \ \mathrm{and} \ d(z_{0},T^{q_x}z_{0})\leq \tilde{\epsilon}/2,
    \end{equation}
    which implies $d(T^{q_{x}}x,z_{0})\leq \tilde{\epsilon}$.
    
    For each integer $0\leq k\leq l_{0}L_{0}-1,$ denote $Y_k=\{x\in G^{\eta}:q_x=k\}.$ Then $G^{\eta}=\cup_{k=0}^{l_{0}L_{0}-1}Y_k.$ Moreover, by Proposition \ref{basic properties}(4), one has $$h^{P}_{top}(T,G^{\eta})\leq \max _{0\leq k\leq l_{0}L_{0}-1} h_{top}^{P}\left(T, Y_k\right).$$ 
    So there exists $k_0\in \left [ 0,l_{0}L_{0}-1\right ]$ such that $$h_{top}^{P}\left(T, Y_{k_0}\right)\geq h^{P}_{top}(T,G^{\eta})\geq \sup\{h(T,\mu):\mu\in K\}-3\eta.$$ 
    Therefore, we can define 
    \begin{equation}\label{equation-DG}
    	G:=\{T^{k_0}x:x\in Y_{k_0}\}.
    \end{equation}
    Meanwhile, since $G_{K}$ is $T$-invariant, one also has $G\subset G_{K}\cap U\cap \{x\in X:C_T(X)\subset \omega_T(x)\}.$ Form the definition of packing entropy, one has $h_{top }^{P}(T, Z)=h_{top}^{P}(T,TZ)$ for any $Z\subseteq X.$ Then $$h^{P}_{top}(T,G_{K}\cap U\cap \{x\in X:C_T(X)\subset \omega_T(x)\})\geq h_{top}^{P}\left(T, G\right)=h_{top}^{P}\left(T, Y_{k_0}\right)\geq \sup\{h(T,\mu):\mu\in K\}-3\eta.$$ 
    By the arbitrariness of $\eta$, this yields $$h^{P}_{top}(T,G_{K}\cap U\cap \{x\in X:C_T(X)\subset \omega_T(x)\})\geq \sup\{h(T,\mu):\mu\in K\},$$ which completes the proof of the first assertion.
\medskip

    {\bf Second assertion:}  If there is an invariant measure with full support, then by Proposition \ref{measure-center}, one has $C_T(X)=X.$ So $\{x\in X:C_T(X)\subset \omega_T(x)\}\subset Tran.$ Combining with the first assertion, we completes the proof of the second assertion.

\begin{remark}
	Using similar method of Theorem \ref{packing-entropy}, we can improve \cite[Theorem 1.4]{HTW} to the following result.
	Suppose that $(X,T)$ satisfies the almost product property and uniform separation property. Then the first assertion of Theorem \ref{packing-entropy} holds if we replace $h_{top}^{P}$ by $h_{top}^{B}$ and replace $\sup$ by $\inf.$ 
\end{remark}

%
%
%
%

\section{ Maximal Lyapunov exponents of Cocycles}\label{sec:proofcocycle}

We will deal with the non-commutative cocycles. As said in \cite{Bar-Gel,Bo} 
that the study of Lyapunov exponents lacks today a satisfactory general approach for non-conformal case, since a complete understanding is just known for some cases such as requiring a clear separation of Lyapunov directions or some number-theoretical properties etc, see \cite{Bar-Gel} and its references therein. An ergodic measure $\mu$ is called Lyapunov-maximizing for cocycle $A$,
if $$\chi_{max}(\mu,A)=\sup_{\nu\in \mathcal{M}(X,T)} \chi_{max}(\nu,A).$$



The (classical) specification  property plays important roles in the study of 
hyperbolic systems and subshifts of finite type (for example, see \cite{Ply,PlSakai,Walters2,DGS}). Here we introduce  their {\it exponential } variants.

\begin{Def}\label{Def-exp-Specification}
A homomorphism $T$ is called to have exponential specification property with exponent $\lambda>0$  (only dependent on the system $T$ itself), if the
following holds: for any $\tau>0$ there exists an integer  $N=N (\tau)>0$ such that for any
$k\geq 1$, any   points $x_1,x_2,\cdots,x_k\in X$,
any integers $a_1\leq b_1<a_2\leq b_2<\cdots<a_k\leq b_k$ with $a_{j+1}-b_j\geq N$ ($1\leq j\leq k-1$),  there exists a point $y\in
X$ such that $d (T^i x_j,T^i y)<\tau e^{-\lambda\min\{i-a_j,b_j-i\}},\,\,a_j\leq i\leq b_j, \,\, 1\leq j\leq k.$



\end{Def}
\begin{remark}
The (classical) specification property  firstly introduced by Bowen  \cite{Bowen3} (or see \cite{KatHas,DGS}) just required that  $d (T^{i} y,T^i x_j)<\tau.$
\end{remark}







\begin{theorem}\label{Thm-Entropy-Maximal-LyapunovIrregular-HolderCocycle}
Let $(X,d)$ be a compact metric space and $T:X \rightarrow X$ be a homeomorphism with exponential specification
property and let $A:X\rightarrow GL(m,\mathbb{R})$ be a H$\ddot{\text{o}}$lder continuous matrix function. Let $U$ be a non-empty open set of $X.$ Then
for any $\mu\in \mathcal{M}^e(X,T)$, the set of $$G_{\mu}\cap \{x\in X :\chi_{max} (A,x)=\chi_{max}(\mu,A)\}\cap U$$ carries full upper capacity entropy.




\end{theorem}

In particular, for Lyapunov-maximizing measure $\mu$, the above set carries full upper capacity entropy even if $\mu$ is supported on a periodic orbit.
\begin{corollary}
	Let $(X,d)$ be a compact metric space and $T:X \rightarrow X$ be a homeomorphism with exponential specification
	property and let $A:X\rightarrow GL(m,\mathbb{R})$ be a H$\ddot{\text{o}}$lder continuous matrix function. Let $U$ be a non-empty open set of $X.$  Then
	for Lyapunov-maximizing measure $\mu$, the set of $G_{\mu}\cap \{x\in X :\chi_{max} (A,x)=\chi_{max}(\mu,A)\}\cap U$ carries full upper capacity entropy.
\end{corollary}


From \cite{Tian2015,Tian2015-2} we know that the subsystem restricted on a topologically mixing locally maximal hyperbolic
set has exponential specification, therefore Theorem \ref{theo:uppercapacityoptimalorbit} will be a direct corollary of Theorem \ref{Thm-Entropy-Maximal-LyapunovIrregular-HolderCocycle}. We will prove Theorem \ref{Thm-Entropy-Maximal-LyapunovIrregular-HolderCocycle} in the rest of this section. Let us begin with a number of useful lemmas.

\subsection{Lyapunov Exponents and Lyapunov Metric}\label{Pesinblock}
Suppose  $T:X\rightarrow X$ to be an invertible map on a compact metric space $X$ and $A:X\rightarrow GL (m,\mathbb{R})$ to  be a  continuous matrix function.

\begin{Def}\label{def2015-Lyapunov-exp}
 For any $x\in X$ and any $0\neq v\in \mathbb{R}^m,$ define the Lyapunov exponent of vector $v$ at $x$,   $\lambda (A,x,v):=\lim_{n\rightarrow +\infty}\frac1n{\log\|A (x,n)v\|},
   $  if the limit exists. We say $x$ to be  (forward) {\it Lyapunov-regular} for $A$,  if $\lambda (A,x,v)$ exists for all vector
    $ v\in \mathbb{R}^m\setminus \{0\}.$ 
    \end{Def}
 \begin{Rem}
   There are many definitions for Lyapunov-regularity, e.g.,   Barreira and Pesin's book \cite{BP} and Ma\~{n}\'{e}'s book \cite{MaBook}.  Definition \ref{def2015-Lyapunov-exp} of Lyapunov-regular point is similar as the one in  \cite{MaBook}, which was originally defined for derivative cocycle. For the definition of  Barreira and Pesin, see \cite{BP} for more details.
\end{Rem}





Now let us recall some Pesin-theoretic techniques, see  \cite[pages 1031-1034]{Kal} (also see \cite{BP}).
Given an ergodic measure $\mu\in \mathcal M^e(X,T) $ and a continuous cocycle $A,$ by  Oseledec Multiplicative Ergodic Theorem  (for example, see \cite [Theorem 3.4.4]{BP} or \cite{Kal}),
 there exist numbers  $\chi_1<\chi_2<\cdots<\chi_t$ (called the {\it Lyapunov exponents} of measure $\mu$),   an $T-$invariant set $\mathcal{R}^\mu$ with $\mu (\mathcal{R}^\mu)=1,$ and an $A-$invariant Lyapunov decomposition of $\mathbb{R}^m$ for $x\in \mathcal{R}^\mu,$
 $\mathbb{R}_x^m=E_{\chi_1} (x)\oplus E_{\chi_2} (x)\oplus \cdots E_{\chi_t} (x)$  with $dim E_{\chi_i} (x)=m_i$ (called called the {\it multiplicity} of the exponent $\chi_i$) such that
for any $i=1,\cdots,t$ and  $0\neq v\in E_{\chi_i} (x)$ one has
 $\lim_{n\rightarrow \pm\infty} \frac1n \log \|A (x,n)v\|=\chi_i.$ 
 The set
$Sp(\mu,A)=\{(\chi_i,m_i):1\leq i \leq t\}$  is the {\it Lyapunov spectrum} of measure $\mu.$
For a fixed $\epsilon> 0$,    there exists a measurable function $K_\epsilon (x)$ and a norm $\|\cdot\|_{x,\epsilon}$  or $\|\cdot\|_{x}$ (called {\it Lyapunov norm}) defined on the set  $\mathcal{R}^\mu$
 such that  for any   point $x\in \mathcal{R}^\mu$,
   \begin{eqnarray}\label{eq-different-norm-estimate}
\|u\|\leq \|u\|_{x,\epsilon}\leq K_\epsilon (x)\|u\|,\,\forall u\in  \mathbb{R}^m,
 \end{eqnarray}
{  For any $l>1$ by Luzin's theorem one  can take the following compact subsets 
\begin{eqnarray}\label{eq-estimate-measure-Pesinblock}
  \mathcal{R}^\mu_{\epsilon,l}\subseteq \{x\in \mathcal{R}^\mu: \,\,K_\epsilon (x)\leq l\}.
 \end{eqnarray}
such that  Lyapunov splitting and Lyapunov norm are continuous on $\mathcal{R}^\mu_{\epsilon,l}$ and   $\lim_{l\rightarrow \infty}\mu (\mathcal{R}^\mu_{\epsilon,l} )= 1.$


\subsection{Estimate of the norm of H$\ddot{\text{o}}$lder cocycles}\label{Estimate-norm}



Let $A$ be an $\alpha-$H$\ddot{\text{o}}$lder cocycle ($\alpha>0$) over a homeomorphism $T$ of a compact metric space $X$ and let $\mu$ be an ergodic measure for $T$ with the maximal Lypunov exponent $\chi_{max}(\mu,A)=\chi.$
From Oseledec Multiplicative Ergodic Theorem (as stated above), it is easy to
see that  $\chi_{max}(\mu,A)=\chi_t$ where $\chi_1<\chi_2<\cdots<\chi_t$ are the Lyapunov exponents of $\mu.$    
Now let us recall two general estimates  from \cite[Lemma 3.1, Lemma 3.3]{Kal} on the norm of $A$ along any orbit segment close to   one orbit of $x\in  \mathcal{R}^\mu$.

The orbit segments $x,Tx,\cdots,T^nx$ and $y,Ty,\cdots,T^ny$ are exponentially $\tau$ close with exponent $\lambda$,
 meaning that  $d (T^i x,T^i y)<\tau e^{-\lambda\min\{i,n-i\}},0\leq i\leq n-1.$   The following lemma follows immediately from \cite [Lemma 3.1] {Kal}.



\begin{Lem}\label{Lem-New-simple-estimate-Lyapunov-new} \cite [Lemma 3.1] {Kal}
For any positive $\lambda$ and $\epsilon$ satisfying $\lambda>\epsilon/ \alpha$ there exists $\tau>0$ such that for any $n\in\mathbb{N}$, any  point $x\in \mathcal{R}^\mu$ with both $x$ and $T^nx$ in $\mathcal{R}^\mu_{\epsilon,l}$, and any point $y\in X$,  if  the orbit segments $x,Tx,\cdots,T^nx$ and $y,Ty,\cdots,T^n (y)$ are exponentially $\tau$ close with exponent $\lambda$,    we have
 \begin{eqnarray}\label{eq-estimate-simple-Lyapunov-new}
  \|A (y,n)\| \leq l e^{l}e^{n (\chi+\epsilon)}\leq l^2e^l e^{2n\epsilon} \|A (x,n)\|.
 \end{eqnarray}

\end{Lem}


%

Another lemma is to estimate the growth of vectors in a certain cone $K\subseteq \mathbb{R}^m$ invariant under $A (x,n)$ \cite[Lemma 3.3]{Kal}.
For any $z\in \mathcal{R}^\mu,$ one has orthogonal splitting $\mathbb{R}^m=E_z\oplus F_z$ with respect to the Lyapunov norm,
 where $E_z$ is the Lyapunov space at $z$ corresponding to the maximal Lyapunov exponent $\chi=\chi_t$ and $F_z$ is the direct sum
 of all other Lyapunov spaces at $z$ corresponding to the Lyapunov exponents less than $\chi$ ($F_z=\{0\}$ if all Lyapunov exponents are same).
  For any vector $u\in \mathbb{R}^m$
   write  $u=u'+u^\perp$  corresponding  to $u'\in E_z$ and $u^\perp\in F_z.$ Let
   $K_{z}=\{u\in \mathbb{R}^m:\,\|u^\bot\|_{z}\leq \|u'\|_{z}\}\,
   $ and 
   $ \,K_{z}^\eta=\{u\in \mathbb{R}^m:\,\|u^\bot \|_{z}\leq  (1-\eta)\|u'\|_{z}\}.$  

\begin{Lem}\label{Lem-simple-estimate-Lyapunov-2}    \cite [Lemma 3.3]{Kal}
There is $\epsilon_0(\mu)>0$ 
 such that for any fixed $\epsilon\in(0,\epsilon_0(\mu))$ and $l\geq 1$, there exist $\eta>0,\,\tau>0$
such that if  $x,T^nx\in\mathcal{R}^\mu_{\epsilon,l}$  and the orbit segments $x,Tx,\cdots,T^nx$ and $y,Ty,\cdots,T^n y$
are exponentially $\tau$ close with exponent $\lambda$, then for every $i=0,1,\cdots,n-1$
 one has $A (T^iy) (K_{T^ix})\subseteq K_{T^{i+1}x}^\eta$ and
$\| (A (y_i)u)'\|_{T^{i+1}x}\geq e^{\chi-2\epsilon}\|u'\|_{T^ix}$ for any $u\in K_{T^ix}.$

\end{Lem}

}
\subsection{Recurrent Time}

In this section we always assume that $T:X\rightarrow X$ is a
continuous map on a compact metric space, $\mu$ is an invariant measure  and $\Gamma$ is an subset of $X$ with positive measure for $\mu$. For $x\in \Gamma,$ define $$t_0(x)=0<t_1(x)<t_2(x)<\cdots$$ to be the all time
such that $T^{t_i}(x)\in\Gamma$ (called recurrent time). By Poincar\'{e} Recurrent Theorem, this definition is well-defined for $\mu$ a.e $x\in \Gamma.$ Note that $$t_1(T^{t_i}x)=t_{i+1}(x)-t_i(x).$$ In general we call $t_1$ to be the first recurrent time.  From \cite[Proposition 3.4]{OT} we know that the recurrent time has a general description.
\begin{Prop}\label{Prop:Rec2} For $\mu$ a.e. $x\in \Gamma,$
$$\lim_{i\rightarrow+\infty}\frac{t_{i+1}(x)}{t_i(x)}=1.$$
In other words,  $$\lim_{i\rightarrow+\infty}\frac{t_{i+1}(x)-t_{i}(x)}{t_i(x)}=0.$$
\end{Prop}

 For  $\Gamma\subseteq X$ with $\mu(\Gamma)>0$ and  $x\in \cup_{j\geq 0} T^{-j}\Gamma,$ define $$t^\Gamma_0(x)<t^\Gamma_1(x)<t^\Gamma_2(x)<\cdots$$ to be the all time
such that $T^{t^\Gamma_i}(x)\in\Gamma$ (called recurrent time). If $x\in \Gamma$, $t^\Gamma_i(x)=t_i(x),$ and moreover if $x\in \cup_{j\geq 0} T^{-j}\Gamma$, then there should exist $j\geq 0$ such that
 $t^\Gamma_i(x)=j+t_i(x).$ Thus by Proposition \ref{Prop:Rec2} we have

\begin{Prop}\label{Prop:Rec33333} For $\mu$ a.e. $x\in \cup_{j\geq 0} T^{-j}\Gamma,$
$$\lim_{i\rightarrow+\infty}\frac{t^\Gamma_{i+1}(x)}{t^\Gamma_i(x)}=1.$$
In other words,  $$\lim_{i\rightarrow+\infty}\frac{t^\Gamma_{i+1}(x)-t^\Gamma_{i}(x)}{t^\Gamma_i(x)}=0.$$
\end{Prop}

In particular, if $\mu$ is ergodic, then $\mu({x\in }\cup_{j\geq 0} T^{-j}\Gamma)=1$. Then the above lemma can be stated almost everywhere. Thus if we considering a sequence of $\Gamma_{m}$ with positive measures, we have

\begin{Prop}\label{Prop:Rec4} Suppose that $\mu(\Gamma_{m})>0$, $m=1,2,\cdots$.  There is  $\mu$ full measure set $X_{*}\subseteq X$ such that for any $x\in X_{*}$ and any $m\geq 1$,
$$\lim_{i\rightarrow+\infty}\frac{t^{\Gamma_m}_{i+1}(x)}{t^{\Gamma_m}_i(x)}=1.$$
In other words,  $$\lim_{i\rightarrow+\infty}\frac{t^{\Gamma_m}_{i+1}(x)-t^{\Gamma_m}_{i}(x)}{t^{\Gamma_m}_i(x)}=0.$$
\end{Prop}

\medskip

We are now ready to prove Theorem \ref{Thm-Entropy-Maximal-LyapunovIrregular-HolderCocycle}.
\subsection{Proof of Theorem \ref{Thm-Entropy-Maximal-LyapunovIrregular-HolderCocycle}}

Let $\epsilon_0(\mu)>0$ be the number satisfying Lemma \ref{Lem-simple-estimate-Lyapunov-2}.
Take a sequence of $\epsilon_{m}<\min{\epsilon_0(\mu),\lambda \alpha}$ with $\lim_{m\rightarrow \infty} \epsilon_{m}=0. $ Then take $l_{m}$ large enough such that  $\mu(\mathcal{R}^\mu_{\epsilon_m,l_m})>0$. For  $\Gamma_{m}=\mathcal{R}^\mu_{\epsilon_m,l_m}$ we  can use Proposition \ref{Prop:Rec4} to get a  $\mu$ full measure set $X_{*}\subseteq X$ such that for any $x\in X_{*}$ and any $m\geq 1$,
$$\lim_{i\rightarrow+\infty}\frac{t^{\Gamma_m}_{i+1}(x)}{t^{\Gamma_m}_i(x)}=1.$$
In other words,  $$\lim_{i\rightarrow+\infty}\frac{t^{\Gamma_m}_{i+1}(x)-t^{\Gamma_m}_{i}(x)}{t^{\Gamma_m}_i(x)}=0.$$ Fix a point $y_{0}\in G_{\mu}\cap X_{*}$.

Let $U$ be a non-empty open set of $X.$ Take $z_0\in X, \delta>0$ such that $B(z_0,\delta)\subseteq U.$ 
Let $Q=G_{\mu}\cap \{x\in X\ :\ \chi_{max} (A,x)=\chi_{max}(\mu,A)\}\cap B(z_0,\delta)$. Next we construct a sequence of closed subsets of $Q$ such that the upper capacity topological entropy  is close to $h_{top}(X,T)$. Fix $\eta>0.$ By the definition of topological entropy, there is  $\epsilon^{*}\in (0,\delta)$,  $n^{*}$ such that for any $\mathcal{N} \geq n^{*}$, there exists a $(\mathcal{N},3\epsilon^{*})$-separated set $\Gamma_{\mathcal{N}}=\{z_{1},z_{2},\dots,z_{r}\} \subseteq X$, such that
\begin{equation*}
\sharp\Gamma_{\mathcal{N}} \geq e^{\mathcal{N}(h_{top}(T)-\eta)}.
\end{equation*}
Let $(X,T)$  have  exponential specification property with exponent $\lambda>0$. Then  for  $\tau_0=\epsilon^{*}>0$ there exists an integer  $N=N (\tau_0)>0$ such that for any
$k\geq 1$, any   points $x_1,x_2,\cdots,x_k\in X$,
 any integers $a_1\leq b_1<a_2\leq b_2<\cdots<a_k\leq b_k$ with $a_{j+1}-b_j\geq N$ ($1\leq j\leq k-1$),  there exists a point $y\in
 X$
such that   $d (T^i x_j,T^i y)<\tau_0 e^{-\lambda\min\{i-a_j,b_j-i\}},\,\,a_j\leq i\leq b_j, \,\, 1\leq j\leq k.$
In particular, for $ x_1,x_2,x_3, n_{1}, n_{2}, n_3\geq 1$, one can take $y\in X$
such that   $d (T^i x_1,T^i y)<\tau_0 e^{-\lambda\min\{i,n_1-1-i\}},\,\,0\leq i\leq n_1-1,$  $d (T^i x_2,T^{i+n_1+N} y)<\tau_0 e^{-\lambda\min\{i,n_2-1-i\}},\,\,0\leq i\leq n_2-1,$
and  $d (T^i x_3,T^{i+n_1+n_2+N}y)<\tau_0 e^{-\lambda\min\{i,2n_3-1-i\}},\,\,
0\leq i\leq  2n_3-1.$
The later implies that $d (T^i x_3,T^{i+n_1+n_2+N} y)<\tau_0 e^{-\lambda i},\,\,0\leq i\leq  n_3-1.$

Now we define the sequences $\{n_{j}'\}$, $\{\epsilon_{j}'\}$, $\{\Gamma_{j}'\}$, inductively by setting:
	\begin{align*}
	&n_{-1}':=1, \epsilon_{-1}':=\epsilon^{*}, \Gamma_{-1}':= \{z_0\} \\\
	&n_{0}':=\mathcal{N}, \epsilon_{0}':=\epsilon^{*}, \Gamma_{0}':= \Gamma_{\mathcal{N}} \ \mathrm{and}\ \mathrm{for} \\
	&j =1,2,3,4,\cdots\\
	&n_{j}':=1,  \Gamma_{j}':= \{T^{j-1}y_{0}\} .
	\end{align*}
	 	For any $s \in \mathbb{N}^{+}$, let
	\begin{equation*}
	G_{s}^{(\mathcal{N})}:= \bar{B}(z_0,\tau_0)\cap\left(\bigcup_{x_{0}\in\Gamma_{0}'}T^{-N}\bar{B}_{n_{0}'}(x_{0},\tau_0)\right) \cap \bigcap_{j=1}^{s}\left(\bigcup_{x_{j}\in\Gamma_{j}'}T^{-M_{j}}\bar{B}_{n_{j}'}(x_{j},\tau_0e^{-\lambda j})\right) \
	\end{equation*} 
    with $M_{j}:=\sum_{l=0}^{j-1}n_{l}'+2N,j=1,2,\cdots.$
	By exponential specification property, $G_{s}^{(\mathcal{N})}$ is a non-empty closed set. Let
$$
G^{(\mathcal{N})}:=\bigcap_{s\geq 1}G_{s}^{(\mathcal{N})}.
$$
One also has $G^{(\mathcal{N})}$ is non-empty and closed set.
	
Next we prove the following:
	\begin{enumerate}
		\item $G^{(\mathcal{N})}\subseteq Q$.
		\item  There exists $Y\subseteq G^{(\mathcal{N})}$ is a $(\mathcal{N}+N,\epsilon^{*})$-separated sets with $\sharp Y=\sharp\Gamma_{\mathcal{N}}$.
	\end{enumerate}
\textbf{Proof of (1)}: Recall $y_0\in G_\mu$. Note that any $y\in G^{(\mathcal{N})}$ satisfies that $d(T^{\mathcal{N}+2N+i}y, T^iy_0)<\tau_0 e^{-\lambda i} $ for any $i$. Thus $\lim_{n\rightarrow \infty}d(T^{i}(T^{\mathcal{N}+2N}y), T^iy_0)=0$ so that $f^{\mathcal{N}+2N}y\in G_\mu$ and then $y\in G_\mu.$

Take $\eta_m>0, \tau_{m}>0$ satisfying Lemma \ref{Lem-simple-estimate-Lyapunov-2} with respect to $\epsilon_m, l_m.$
Take $\tau_{m}>0$ small   if necessary such that it also satisfies Lemma \ref{Lem-New-simple-estimate-Lyapunov-new} with respect to $\epsilon_m.$ Fix $m$ and  $y\in G^{(\mathcal{N})}$. Take $J$ large enough such that $\tau_0 <\tau_m  e^{\lambda J}  $ so that $d(T^{\mathcal{N}+2N+i}y, T^iy_0)<\tau_0 e^{-\lambda i} <\tau_m e^{-\lambda (i-J)} $ for any $i\geq J$.

 By choice of $y_{0}$, take $S$ large enough such that for any $s\geq S,$ $t^{\Gamma_m}_s(y_0)\geq J.$  Note that for $s>S$, the orbit segments $$z=T^{t^{\Gamma_m}_S(y_0)}y_0,Tz,\cdots, T^{t^{\Gamma_m}_s(y_0)-t^{\Gamma_m}_S(y_0)}z$$
  and $p=T^{\mathcal{N}+2N+t^{\Gamma_m}_S(y_0)}y,Tp, \cdots, T^{\mathcal{N}+2N+t^{\Gamma_m}_s(y_0)-t^{\Gamma_m}_S(y_0)}p$
are exponentially $\tau_{m}$ close with exponent $\lambda$. Then by Lemma \ref{Lem-New-simple-estimate-Lyapunov-new} we have $$
  \|A (p,t^{\Gamma_m}_s(y_0)-t^{\Gamma_m}_S(y_0))\|  \leq l_m^2e^{l_m} e^{2(t^{\Gamma_m}_s(y_0)-t^{\Gamma_m}_S(y_0))\epsilon_m} \|A (z,t^{\Gamma_m}_s(y_0)-t^{\Gamma_m}_S(y_0))\|.$$ Thus
   $$\limsup_{s\rightarrow \infty}\frac{1}{t^{\Gamma_m}_s(y_0)}\|A (y,t^{\Gamma_m}_s(y_0))\|
   =\limsup_{s\rightarrow \infty}\frac{1}{t^{\Gamma_m}_s(y_0)}\|A (p,t^{\Gamma_m}_s(y_0))\|
   =\limsup_{s\rightarrow \infty}\frac{1}{t^{\Gamma_m}_s(y_0)}\|A (p,t^{\Gamma_m}_s(y_0)-t^{\Gamma_m}_S(y_0))\| $$$$\leq 2\epsilon_m+ \limsup_{s\rightarrow \infty}\frac{1}{t^{\Gamma_m}_s(y_0)}\|A (z,t^{\Gamma_m}_s(y_0)-t^{\Gamma_m}_S(y_0))\|
    =2\epsilon_m+\limsup_{s\rightarrow \infty}\frac{1}{t^{\Gamma_m}_s(y_0)}\|A (z,t^{\Gamma_m}_s(y_0))\|
   $$ $$ =2\epsilon_m+\limsup_{s\rightarrow \infty}\frac{1}{t^{\Gamma_m}_s(y_0)}\|A (y_0,t^{\Gamma_m}_s(y_0))\| =2\epsilon_m+\chi
.$$

   By Lemma \ref{Lem-simple-estimate-Lyapunov-2} we have $\| (A (p_i)u)'\|_{f^{i+1}z}\geq e^{\chi-2\epsilon}\|u'\|_{f^iz}$ for any $u\in K_{f^iz}$ which implies that $ \| (A (p,t^{\Gamma_m}_s(y_0)-t^{\Gamma_m}_S(y_0) )u)'\|_{z}\geq e^{(t^{\Gamma_m}_s(y_0)-t^{\Gamma_m}_S(y_0))(\chi-2\epsilon)}\|u'\|_{z}.$ By (\ref{eq-different-norm-estimate}) and (\ref{eq-estimate-measure-Pesinblock}),  we have  $$ \| (A (p,t^{\Gamma_m}_s(y_0)-t^{\Gamma_m}_S(y_0) )u)'\|_{z}\geq \frac 1{l_m} \| (A (p,t^{\Gamma_m}_s(y_0)-t^{\Gamma_m}_S(y_0) )u)'\|$$$$\geq \frac 1{l_m}e^{(t^{\Gamma_m}_s(y_0)-t^{\Gamma_m}_S(y_0))(\chi-2\epsilon)}\|u'\|_{z}\geq  \frac 1{l_m}e^{(t^{\Gamma_m}_s(y_0)-t^{\Gamma_m}_S(y_0))(\chi-2\epsilon)}\|u'\|.$$
Similarly, one can get  $$\liminf_{s\rightarrow \infty}\frac{1}{t^{\Gamma_m}_s(y_0)}\|A (y,t^{\Gamma_m}_s(y_0))\|
  \geq \chi-2\epsilon_m.$$
  Recall that $\lim_{i\rightarrow+\infty}\frac{t^{\Gamma_m}_{i+1}(x)}{t^{\Gamma_m}_i(x)}=1.$ Thus $$\epsilon_m+\chi\geq \limsup_{s\rightarrow \infty}\frac{1}{t^{\Gamma_m}_s(y_0)}\|A (y,t^{\Gamma_m}_s(y_0))\|
= \limsup_{q\rightarrow \infty}\frac{1}{q}\|A (y,q)\| $$$$\geq  \liminf_{q\rightarrow \infty}\frac{1}{q}\|A (y,q)\|
=
\liminf_{s\rightarrow \infty}\frac{1}{t^{\Gamma_m}_s(y_0)}\|A (y,t^{\Gamma_m}_s(y_0))\|
  \geq \chi-2\epsilon_m.$$
  Since $\epsilon_m$ goes to zero, we have $\lim_{q\rightarrow \infty}\frac{1}{q}\|A (y,q)\|=\chi.$


	\textbf{Proof of (2)}: For any $z_{i} \in \Gamma_{\mathcal{N}}$, $G^{i}:= \bar{B}(z_0,\tau_0)\cap T^{-N}\bar{B}_{\mathcal{N}}(z_{i},\epsilon^{*})\cap (\bigcap_{j\geq 1}(\bigcup_{x_{j}\in\Gamma_{j}'}T^{-M_{j}}\bar{B}_{n_{j}'}(x_{j},\tau_0 e^{-\lambda j})))$ is nonempty and closed, then $G^{(\mathcal{N})}=\cup_{i=1}^{r} G^{i}$. Take $y_{i}\in G^{i}$ for $i=1,\dots,r$, it is easy to check that $Y:=\{y_{1},\dots,y_{r}\} \subseteq \cup_{i=1}^{r} G^{i} =G^{(\mathcal{N})}$ is a $(\mathcal{N}+N,\epsilon^{*})$-separated sets, since $\Gamma_{\mathcal{N}}$ is   $(\mathcal{N},3\epsilon^{*})$-separated.  We get that $Y:=\{y_{1},\dots,y_{r}\} $ is a $(\mathcal{N}+N,\epsilon^{*})$-separated sets for $G^{(\mathcal{N})}$ with $\sharp Y=\sharp\Gamma_{\mathcal{N}}\geq e^{\mathcal{N}(h_{top}(T,X)-\eta)}$, thus by arbitrariness of $\mathcal{N},$ $h_{top}^{UC}(T,Q)\geq  h_{top}(T,X)-\eta$.
\qed

\section{Complexity of optimal orbits of Birkhoff averages}\label{sec:regularity sensitivity result}
This section concentrates on the optimal orbits of Birkhoff averages, and is divided into three subsections. In section 6.1 we will prove Theorem \ref{theo:regularity sensitivity result} while in section 6.2, we will prove Theorem \ref{theorem-chaos}. We will also prove analogous results for the set of irregular points in Section 6.3.
\subsection{Proof of Theorem 1.4} Let us begin with some discussions of the set of measure recurrent optimal orbits $S_{f}^{MR}.$
\begin{proposition}\label{prop:existence of MR}\cite{J1}
  For each TDS $(X,T)$, suppose $f:X\to \mathbb{R}$ is continuous, then $S_{f}^{MR}\neq\emptyset$.
\end{proposition}
The proof is given in \cite{J1}, but since it is very simple, and is the basis for future discussion, we put it here for completeness.
\begin{proof}
Let $\beta(f):=\sup_{x\in X}\limsup_{n\to\infty}\frac{1}{n}f^{(n)}(x)$. Due to the Birkhoff's ergodic theorem, we will have
\begin{equation*}
  \beta(f)=\sup_{\mu\in \cM(X,T)}\int fd\mu. 
\end{equation*}
Note that the operator $\int fd(\cdot):\cM(X,T)\to\mathbb{R}$ is continuous with respect to weak topology. So there must be a measure $\mu\in \cM(X,T)$ attains the supremum. Together with ergodic decomposition theorem, we can actually assume $\mu$ to be ergodic. Then for $\mu$-a.e-$x_{0}$, $\langle f\rangle(x_{0})=\beta(f)$. In other words, the measure-recurrent optimal orbit always exists.
\end{proof}
From the proof of Proposition \ref{prop:existence of MR}, the structure of $S_{f}^{MR}$ is equivalent to the structure of the maximizing measure $\mu$, i.e., $\int fd\mu=\beta(f)$. For an Anosov diffeomorphism, the structure of the $f$-maximizing measure varies when the function $f$ admits different regularity. To sum up,
\begin{lemma}\cite{Bousch,Bousch-Jenkinson,Bremont,Morris}\label{lem:continous}
There exists a Baire generic subset $\mathcal{F}$ in the space of continuous functions such that for any $f\in \mathcal{F}$, the $f$-maximizing measure is unique, fully supported on $X$, and has zero entropy.
\end{lemma}
\begin{lemma}\cite{Con16,HLMXZ191}\label{lem:holdercontinous}
There exists an open and dense subset $\mathcal{G}$ in the space of H\"{o}lder continuous or $C^{1}$ smooth functions such that for any $f\in\mathcal{G}$, the $f$-maximizing measure is unique and supported on a periodic orbit.
\end{lemma}
\begin{lemma}\cite[Theorem A]{Shinoda18}\label{lem:continous2}
	There exists a dense subset $\mathcal{H}$ in the space of continuous functions such that for any $f\in\mathcal{H},$ there are uncountably many maximizing measures whose support are full and whose entropy is positive.
\end{lemma}

Based on the above facts, we are ready to prove Theorem \ref{theo:regularity sensitivity result}.

\noindent\textbf{Proof of Theorem \ref{theo:regularity sensitivity result}}: 
For any continuous function $f,$ let $\mu\in \mathcal M(X,T)$ be an $f$-maximizing measure, then $G_{\mu}\subseteq S^{op}_{f}$ and $G_{\mu}^{T}\subseteq S_{f}^{TR}.$ By Theorem \ref{entropy}, we complete the first assertion of Theorem \ref{theo:regularity sensitivity result}.

By Lemma \ref{lem:continous}, there exists a Baire generic subset $\mathcal{F}$ in the space of continuous functions such that for any $f\in \mathcal{F}$, one has $F(\beta(f))=\{\mu\}$, and $supp(\mu)=X,~h(T,\mu)=0$. Thus $R_f^{T}(\beta(f))=G^{T}_{\mu}= S_{f}^{TR}=S_{f}^{MR}=S^{op}_{f}.$
On the other hand, by using Lemma \ref{lem:holdercontinous}, there exists an open and dense subset $\mathcal{G}$ in the space of H\"{o}lder continuous or $C^{1}$ smooth functions such that for any $f\in\mathcal{G}$, one has $F(\beta(f))=\{\mu_{0}\}$, where $\mu_{0}$ is a periodic measure supported on $\mathrm{orb}(x_0,T)$. This implies $h(T,\mu_{0})=0$, $R_f^{T}(\beta(f))=G^{T}_{\mu_{0}}=S_{f}^{TR},$ $S_{f}^{MR}=\cup_{i=0}^{\infty}T^{-i}\mathrm{orb}(x_0,T)$ and $S^{op}_{f}=G_{\mu_{0}}$. Then $h_{top}^{UC}(T,S_{f}^{MR})=h_{top}^{UC}(T,\cup_{i=0}^{\infty}T^{-i}\mathrm{orb}(x_0,T))=0.$ By Theorems \ref{packing-entropy} and  \cite[Theorem 1.4]{HTW} we complete the second assertion of Theorem \ref{theo:regularity sensitivity result}.

By Lemma \ref{lem:continous2}, there exists a dense subset $\mathcal{H}$ in the space of continuous functions such that for any $f\in\mathcal{H},$ there exists an $f$-maximizing measure $\mu\in \mathcal M(X,T)$ whose support is full and whose entropy is positive. Note that $G_{\mu}^{T}\subseteq S_{f}^{TR}$ and $G_{\mu}^{T}\subseteq S_{f}^{MR},$ we complete the third assertion of Theorem \ref{theo:regularity sensitivity result} by Theorems \ref{packing-entropy} and \cite[Theorem 1.4]{HTW}.\qed

\subsection{Proof of Theorem \ref{theorem-chaos}}
In this subsection, we will prove the chaotic behaviors in $S_{f}^{TR}$. 
For a TDS $(X,T),$ we say a pair $p, q \in X$ is \emph{distal} if $\lim \inf _{i \rightarrow \infty} d$ $\left(T^{i} p, T^{i} q\right)>0 .$ 
\begin{lemma}\cite[Theorem F]{CT}\label{lemma-specification}
	Suppose that $(X,T)$ is a transitive Anosov diffeomorphism on a compact manifold and let $K$ be a connected non-empty compact subset of $\mathcal M(X,T)$. If there is a $\mu\in K$ such that $\mu=\theta\mu_1+(1-\theta)\mu_2\ (\mu_1=\mu_2\mathrm{\ could\ happens})$ where $\theta\in[0,1]$, and $G_{\mu_1}$, $G_{\mu_2}$ both have distal pair, then for any non-empty open set $U\subseteq X$, there exists an uncountable DC1-scrambled set $S_K\subseteq G_K\cap U\cap Tran$.
\end{lemma}

\begin{lemma}\label{generic distal}\cite[Lemma 4.1]{CT}
	Suppose that $\mu\in\mathcal{M}^e(X,T)$, $S_\mu$ is nondegenerate and minimal. Then, $G_\mu$ has distal pair.
\end{lemma}

\begin{proposition}\label{proposition-minimal}
	Suppose that $(X,T)$ satisfies entropy-dense property. If $(X,T)$ is not uniquely ergodic, then there exist minimal subsets $\Lambda_{1},\Lambda_{2}\subset X$ such that $\Lambda_{1}\cap\Lambda_{2}=\emptyset$.
\end{proposition}
\begin{proof}
	Since $(X,T)$ is not uniquely ergodic, there exist $\nu_{1},\nu_{2}\in \mathcal{M}(X,T)$ such that $\nu_{1}\neq \nu_{2}$. Let $\varepsilon=d(\nu_{1},\nu_{2})$. By entropy-dense property, there exist closed $T$-invariant sets $\Lambda_{\nu_{1}},\Lambda_{\nu_{2}}\subseteq X$, such that  $\mathcal{M}(T,\Lambda_{\nu_{1}})\subseteq B(\nu_{1},\frac{\varepsilon}{3})$ and $\mathcal{M}(T,\Lambda_{\nu_{2}})\subseteq B(\nu_{2},\frac{\varepsilon}{3})$. Then $\Lambda_{\nu_{2}}\cap \Lambda_{\nu_{2}}=\emptyset$, otherwise, there exists $\nu  \in\mathcal{M}(T,\Lambda_{\nu_{1}}\cap\Lambda_{\nu_{2}})$, then $d(\nu,\nu_{1})<\frac{\varepsilon}{3}$ and $d(\nu,\nu_{2})<\frac{\varepsilon}{3}$. Take a minimal subset $\Lambda_{i}\subset\Lambda_{\nu_{i}}$ for $i=1,2$, then $\Lambda_{1}\cap\Lambda_{2}=\emptyset$.
\end{proof}
\begin{lemma}[\cite{HTW}, Lemma 3.5]\label{AP}
	Suppose that $(X,T)$ has almost product property and there is some invariant measure $\mu$ with full support, then the almost periodic set AP is dense in X.
\end{lemma}
Next, we prove Theorem \ref{theorem-chaos}.

\noindent\textbf{Proof of Theorem \ref{theorem-chaos}}: 
    By Lemma \ref{lem:holdercontinous}, there exists an open and dense subset $\mathcal{G}$ in the space of H\"{o}lder continuous or $C^{1}$ smooth functions such that for any $f\in\mathcal{G}$, the $f$-maximizing measure is unique and supported on a periodic orbit $\mathrm{orb}(x_0,T)$, we denote the measure by $\mu_{0}$. 
    
    If $\sharp\mathrm{orb}(x_0,T)>1$, then $\mu_{0}$ is ergodic, $S_{\mu_{0}}$ is nondegenerate and minimal. Thus $G_{\mu_0}$ has a distal pair by Lemma \ref{generic distal}. So there exists an uncountable DC1-scrambled set $S\subseteq G_{\mu_0}\cap Tran\cap U=S_{f}^{TR}\cap U$ by Lemma \ref{lemma-specification}. Therefore we complete the first assertion of the Theorem \ref{theorem-chaos}.

	Otherwise, if $\sharp\mathrm{orb}(x_0,T)=1,$ it is enough to prove that for any $x_{0}\in X$ with $f(x_{0})=x_{0}$ and any non-empty open set $U\subseteq X,$ there exists an uncountable set $S\subseteq G_{\mu_{0}}^{T}\cap U$ (see (\ref{equation-DK})) such that $S$ is chaotic in the sense of (\ref{equation-DE}) where $\mu_{0}=\delta_{x_{0}}$.
	
	To begin with, since $(X,T)$ satisfies almost product property, we can let $g:\mathbb{N}\rightarrow\mathbb{N}$ and $m:\mathbb{R}^{+}\rightarrow \mathbb{N}$ be the two maps in the definition. Meanwhile, since there is some invariant measure $\mu$ with full support, then by using Proposition \ref{proposition-minimal} there exist minimal subsets $\Lambda_{1},\Lambda_{2}\subset X$ such that $\Lambda_{1}\cap\Lambda_{2}=\emptyset.$ Take $x_{i}\in \Lambda_{i}$ for $i\in\{1,2\}$. Denote $\epsilon^{*}=d(\Lambda_{1},\Lambda_{2})>0.$
	
	By Lemma \ref{AP} for any non-empty open set $U$, we can fix an $\tilde{\epsilon}>0$, a point $z_{0}\in AP$ 
	and $L_{0}\in\mathbb{N}$ such that $\overline{B(z_{0},\tilde{\epsilon})}\subseteq U,$ and for any $l\geq 1$, there is $p\in \left[ l,l+L_{0}\right] $ such that $f^{p}(z_{0})\in B(z_{0},\tilde{\epsilon}/2).$ This implies that
	\begin{equation}\label{equation-Ak}
		\frac{\sharp\{0\leq p \leq lL_{0}:d(T^{p}(z_{0}),z_{0}) \leq \tilde{\epsilon}/2\}}{lL_{0}}\geq \frac{1}{L_{0}}.
	\end{equation}
	Take $l_{0}$ large enough such that
	\begin{equation}\label{AM}
		l_{0}L_{0}\geq m(\tilde{\epsilon}/2),\  \frac{g(l_{0}L_{0})}{l_{0}L_{0}}<\frac{1}{4L_{0}}.
	\end{equation}
	
	Let $\{\zeta_{k}\}$ and $\{\epsilon_{k}\}$ be two strictly decreasing sequences so that $\lim_{k\rightarrow \infty}\zeta_{k}=\lim_{k\rightarrow \infty}\epsilon_{k}=0$ with $\epsilon_{1}<\frac{1}{4}\epsilon^{*}$.
	
	By Lemma \ref{AP} the almost periodic set AP is dense in $X$. Thus, for any fixed $k$ there is a finite set $\Delta_{k}:=\{x_{1}^{k},x_{2}^{k},\dots,x_{t_{k}}^{k}\}\subseteq AP$ 
	and $L_{k}\in\mathbb{N}$ such that $\Delta_{k}$ is $\epsilon_{k}$-dense in $X,$ and for any $1\leq i\leq t_{k}$, any $l\geq 1$, there is $p_{i}\in \left[ l,l+L_{k}\right] $ such that $T^{p_{i}}x_{i}^{k}\in B(x_{i}^{k},\epsilon_{k})$. This implies that any $1\leq i\leq t_{k}$,
	\begin{equation}\label{da}
		\frac{\sharp\{0\leq p_{i} \leq lL_{k}:d(T^{p_{i}}x_{i}^{k},x_{i}^{k})<\epsilon_{k}\}}{lL_{k}}\geq \frac{1}{L_{k}}.
	\end{equation}
	Take $l_{k}$ large enough such that
	\begin{equation}\label{db}
		l_{k}L_{k}\geq m(\epsilon_{k}), \frac{g(l_{k}L_{k})}{l_{k}L_{k}}<\frac{1}{4L_{k}}.
	\end{equation}
    Take $\tilde{n}_{k}>4$ large enough such that
    \begin{equation}\label{dc}
    	\tilde{n}_{k}\geq m(\epsilon_{k}), \frac{g(\tilde{n}_{k})}{\tilde{n}_{k}}<\frac{1}{4}.
    \end{equation}
    We choose $\{n_{k}\}$, with $n_{k}\in\mathbb{N}$, such that
    \begin{equation}\label{de}
    	\frac{g(n_{k})}{n_{k}}\leq \epsilon_{k},\ n_{k}\geq m(\epsilon_{k}),\ \frac{t_{k}l_{k}L_{k}+k\tilde{n}_{k}}{n_{k}}\leq \zeta_{k}.
    \end{equation}
    We then choose a strictly increasing $\{N_{k}\}$, with $N_{k}\in\mathbb{N}$, such that
    \begin{equation}\label{dl}
    	n_{k+1}+(k+1)\tilde{n}_{k+1}+t_{k+1}l_{k+1}L_{k+1}\leq \zeta_{k}\left(\sum_{j=1}^{k}(n_{j}N_{j}+j\tilde{n}_{j}+t_{j}l_{j}L_{j})+l_{0}L_{0}\right)
    \end{equation}
    and
    \begin{equation}\label{dm}
    	\sum_{j=1}^{k-1}\left(n_{j}N_{j}+j\tilde{n}_{j}+t_{j}l_{j}L_{j}\right)+l_{0}L_{0}\leq \zeta_{k}n_{k}N_{k}.
    \end{equation}

    Now, giving an $\xi=(\xi_1,\xi_2,\cdots)\in\{1,2\}^\infty$, we construct the $z_\xi$.
    Define the sequences $\{n_{j}'\}$, $\{\epsilon_{j}'\}$, $\{\Gamma_{j}'\}$ inductively, by setting for:
    \begin{align*}
    	&j = 0,\\
    	&n_{0}':=l_{0}L_{0}, \epsilon_{0}':=\tilde{\epsilon}/2, \Gamma_{0}':= \{z_0\} \ \mathrm{and}\ \mathrm{for} \\
    	&j = N_{1}+N_{2}+\dots+N_{k-1}+1+2+\dots+k-1+t_{1}+\dots+t_{k-1}+q \ \mathrm{with}\ 1\leq q \leq N_k,\\
    	&n_{j}':=n_{k}, \epsilon_{j}':=\epsilon_{k}, \Gamma_{j}':= \{x_{0}\} \ \mathrm{and}\ \mathrm{for} \\
    	&j = N_{1}+N_{2}+\dots+N_{k}+1+2+\dots+k-1+t_{1}+\dots+t_{k-1}+q \ \mathrm{with}\ 1\leq q \leq k,\\
    	&n_{j}':=\tilde{n}_{k}, \epsilon_{j}':=\epsilon_{k}, \Gamma_{j}':= \{x_{\xi_{q}}\} \ \mathrm{and}\ \mathrm{for} \\
    	&j = N_{1}+N_{2}+\dots+N_{k}+1+2+\dots+k+t_{1}+\dots+t_{k-1}+q \ \mathrm{with}\ 1\leq q \leq t_k,\\
    	&n_{j}':=l_{k}L_{k}, \epsilon_{j}':=\epsilon_{k}, \Gamma_{j}':= \{x_{q}^{k}\}.
    \end{align*}
    For each $s \in \mathbb{N}^{+}$, define
    \begin{equation*}
    	G_{s}^{\xi}:= \bigcap_{j=0}^{s}\left(\bigcup_{x_{j}\in\Gamma_{j}'}T^{-M_{i-1}}B_{n_{j}'}(g;x_{j},\epsilon_{j}')\right),
    \end{equation*}
    where $M_{j}:=\sum_{l=0}^{j}n_{l}'$ for any $j\in\mathbb{N}$ and $M_{-1}:=0.$
    
    By almost product property, $G_{s}^{\xi}$ is a non-empty closed set. Let
    $$
    G^{\xi}:=\bigcap_{s\geq 1}G_{s}^{\xi}.
    $$
    By the nested structure of $G_{s}^{\xi},$ one also has $G^{\xi}$ is non-empty and closed set. Let $z_{\xi}\in G^{\xi}$ and let $Y=\{z_{\xi}:\xi\in\{1,2\}^\infty\}$. Using the same method in the proof of Items $(1)$ and $(2)$ of Theorem \ref{entropy}, we have $Y\subseteq G^{T}_{\mu_{0}}$. Next we prove that $Y$ is chaotic in the sense of (\ref{equation-DE}).
	Fix $\xi \neq \eta\in\{1,2\}^\infty$, suppose $\xi_{u}\neq\eta_{u}$ for some $u\in\mathbb{N}$. By construction of $G^{\xi}$ and $G^{\eta}$, for any fixed $k\geq u$, there is $a=a_{k}$ such that there is $\Lambda^{k}_{\xi},\Lambda^{k}_{\eta}\subseteq \Lambda_{\tilde{n}_{k}},$
	\begin{equation}
		\max\{d(T^{a+l}z_{\xi},T^{l}x_{\xi_{u}}): l\in \Lambda^{k}_{\xi}\}\leq \epsilon_{k}\ \text{and}\ 	\max\{d(T^{a+l}z_{\eta},T^{l}x_{\eta_{u}}): l\in \Lambda^{k}_{\eta}\}\leq \epsilon_{k}.
	\end{equation}
	By (\ref{dc}), one has $\Lambda^{k}_{\xi}\cap \Lambda^{k}_{\eta}\neq\emptyset$.
	Together with $\epsilon_{k}<\frac{1}{4}\epsilon^{*}$, we get that  there is $p_{k}\in \left [ 0,\tilde{n}_{k}-1\right ]$ such that
	\begin{equation}
		d(T^{a+p_{k}}z_{\xi},T^{a+p_{k}}z_{\eta})\geq \frac{1}{2}\epsilon_{*}.
	\end{equation}
	This implies $\limsup_{n\to +\infty}d(T^{n}z_{\xi},T^{n}z_{\eta})\geq \frac{1}{2}\epsilon^{*}>0$ and $Y$ is uncountable. On the other hand, for any fixed $t>0$, we can choose $k_t\in\mathbb{N}$ large enough such that $2\epsilon_{k}<t$ holds for any $k\geq k_t$. Note that $T^{b(k)}z_{\xi},T^{b(k)}z_\eta \in \bigcap_{j=1}^{N_k}T^{-(j-1)n_k} B_{n_{k}}(g;x_{0},\epsilon_{k})$, where
$$
b(k)=l_{0}L_{0}+\sum_{j=1}^{k-1}t_{j}l_{j}L_{j}+\sum_{j=1}^{k-1}j\tilde{n}_{j}+\sum_{j=1}^{k-1}N_{j}n_{j}.
$$
One thus has
	\begin{align*}
		&\limsup_{n\to \infty}\frac{1}{n}|\{j\in [0,n-1]:\ d(f^jz_\xi,f^jz_\eta)<t\}|\\
		\ge & \limsup_{n\to \infty}\frac{1}{n}|\{j\in [0,n-1]:\ d(f^jz_\xi,f^jz_\eta)<2\epsilon_{k}\}|\\
		\ge & \limsup_{k\geq k_t,\ k\to \infty}\frac{|\{j\in [0,b(k)+N_kn_{k}-1]:\ d(f^jz_\xi,f^jz_\eta))<2\epsilon_{k}\}|}{b(k)+N_kn_{k}}\\
		\ge & \limsup_{k\geq k_t,\ k\to \infty}\frac{N_kn_{k}}{b(k)+N_kn_{k}}(1-2\frac{g(n_{k})}{n_{k}})\\
		\ge & \limsup_{k\geq k_t,\ k\to \infty}\frac{1-2\zeta_{k}}{1+\zeta_{k}}\quad &(\text{using}\ (\ref{de})\ \text{and}\ (\ref{dm}))\\
		= & 1.
	\end{align*}
	So $Y$ is uncountable and is chaotic in the sense of (\ref{equation-DE}). 
	
	By construction of $Y$ and (\ref{AM}), for any $\xi=(\xi_1,\xi_2,\cdots)\in\{1,2\}^\infty$, there is $\Lambda^{\xi}\subseteq \Lambda_{l_{0}L_{0}}$ such that
	\begin{equation}
		\max\{d(T^{l}z_{\xi},T^{l}z_{0}):l\in \Lambda^{\xi}\}\leq \tilde{\epsilon}/2,\
		\frac{\sharp\Lambda^{\xi}}{l_{0}L_{0}}\geq 1-\frac{g(l_{0}L_{0})}{l_{0}L_{0}}\geq 1-\frac{1}{4L_{0}}.
	\end{equation}
	Together with (\ref{equation-Ak}) we get that there is $q_{\xi}\in \left [ 0,l_{0}L_{0}-1\right ]$ such that
	\begin{equation}
		d(T^{q_{\xi}}z_{\xi},T^{q_{\xi}}z_{0})\leq \tilde{\epsilon}/2 \ \mathrm{and} \ d(z_{0},T^{q_{\xi}}z_{0})\leq \tilde{\epsilon}/2,
	\end{equation}
	which implies $d(T^{q_{\xi}}z_{\xi},z_{0})\leq \tilde{\epsilon}$.
	
	Using the pigeon-hole principle, we obtain that there is an uncountable subset $Y_1\subseteq Y$ and $q \in \left [ 0,l_{0}L_{0}-1\right ]$ such that  $d(T^{q}z_{\xi},z_{0})\leq \tilde{\epsilon}$ for any $z_{\xi} \in Y_1.$ Therefore, we can define
	\begin{equation}\label{equation-DK}
		S:= \{f^{q}z_{\xi}:z_{\xi} \in Y_{1}\}.
	\end{equation} 
    Meanwhile, since $G_{\mu_{0}}^{T}$ is $T$-invariant, one also has $S\subset G_{\mu_{0}}^{T}\cap U.$
	And we can verify that $S$ is uncountable and is chaotic in the sense of (\ref{equation-DE}).
	
	Finally, since $S_{f}^{MR}=\cup_{i=0}^{\infty}T^{-i}\mathrm{orb}(x_0,T),$ it is easy to check that there is no Li-Yorke pair in $S_{f}^{MR}.$ 
	This completes the proof of Theorem \ref{theorem-chaos}.\qed

\subsection{Complexity of Level set}
In this subsection, we will give a result about the complexity of level set. From \cite{Tho2012}, when $(X,T)$ satisfies the almost product property, the irregular set is either empty or has full topological entropy. So we have that either the irregular set is empty or the packing entropy and upper capacity entropy of the irregular set cariess full topological entropy. Here for a continuous function $f$ on $X$, the \emph{$f-$irregular set} is
$\left\{x\in X:\lim_{n\to\infty}\frac1n\sum_{i=0}^{n-1}f(T^ix) \,\, \text{ diverges }\right\}.$
However, the three entropies of level set are different. Denote
$$
L_f=\left[\inf_{\mu\in \mathcal M(X,T)}\int f d\mu,  \,  \sup_{\mu\in \mathcal M(X,T)}\int f d\mu\right].
$$
For any $a\in  L_f,$ define the level set
\begin{equation}
	R_{f}(a) := \left\{x\in X: \lim_{n\to\infty}\frac1n\sum_{i=0}^{n-1}f(T^ix)=a\right\}.
\end{equation}
In particular,
$$
S_{f}^{op}=R_{f}(\beta(f))
$$
is a level set at boundary. In Theorem 1.4, we prove that three entropies of $R_{f}(\beta(f))$ are different, now we will show similar phenomenon on $R_{f}(a)$ for any  $a\in  L_f.$
\begin{theorem}\label{theo-level}
	Suppose that $(X,T)$ is a transitive Anosov diffeomorphism on a compact manifold. Let $f:X\rightarrow \mathbb{R}$ be a continuous function. Then for any $a\in  L_f$ and any non-empty open set $U\subseteq X,$ we have
	\begin{itemize}
		\item $h_{top}^{B}(T,R_f^{T}(a)\cap U)=h_{top}^{B}(T,R_f(a)\cap U)=h_{top}^{B}(T,R_f^{T}(a))=h_{top}^{B}(T,R_f(a))=t_{a}$;
		\item $h_{top}^{P}(T,R_f^{T}(a)\cap U)=h_{top}^{P}(T,R_f(a)\cap U)=h_{top}^{P}(T,R_f^{T}(a))=h_{top}^{P}(T,R_f(a))= t_{a}$;
		\item $h_{top}^{UC}(T,R_f^{T}(a)\cap U)=h_{top}^{UC}(T,R_f(a)\cap U)=h_{top}^{UC}(T,R_f^{T}(a))=h_{top}^{UC}(T,R_f(a))=h_{top}(T,X),$
	\end{itemize}
	where $t_{a}=\sup_{\mu\in \mathcal M(X,T)}\left\{h(T,\mu):\,  \int f d\mu=a\right\}$ and $R_f^{T}(a)=R_f(a)\cap Tran$. In particular, let $\mu_{\max}$ be the unique measure of maximal entropy, i.e., $h(T,\mu_{\max})=h_{top}(T,X),$ then for any $a\in  L_f\setminus \{\int fd\mu_{\max}\},$ one has $$h_{top}^{B}(T,R_f(a))=h_{top}^{P}(T,R_f(a))<h_{top}(T,X)=h_{top}^{UC}(T,R_f^{T}(a)\cap U).$$
\end{theorem}
\begin{proof}
	It is clear that
	$R_f(a)=\{x\in X:M_{x}\subseteq F(a)\}$,
	where $F(a):=\{\rho\in\cM(X,T):\int fd\rho=a\}$. $h_{top}^{B}(T,R_f(a))=t_{a}$ is proved by Pfister and Sullivan in \cite[Proposition 7.1]{PS}. For any $\mu$ with $\int f d\mu=a$, one has $G_{\mu}^{T}\subseteq R_f^{T}(a)$, then by \cite[Theorem 1.4]{HTW} one has $h_{top}^{B}(T,R_f^{T}(a))\geq h_{top}^{B}(T,G_{\mu}^{T})=h(T,\mu).$ So we obtain $h_{top}^{B}(T,R_f^{T}(a))=t_{a}.$ Since $R_f(a)$ is $T-$invariant, one has $h_{top}^{B}(T,R_f^{T}(a)\cap U)=h_{top}^{B}(T,R_f(a)\cap U)=t_{a}$ by \cite[Lemma 3.12]{HTW}. 
	
	By checking the Part II \cite[Page 398]{ZCC2012} and \cite[Corollary 3.4]{ZCC2012}, we get that for any TDS $(X,T)$ without additional hypothesis, $h_{top}^{P}(T,R_f^{T}(a))\leq h_{top}^{P}(T,R_f(a))\leq t_{a}$ holds. Combining with Theorems \ref{packing-entropy}, we has $h_{top}^{P}(T,R_f^{T}(a))=h_{top}^{P}(T,R_f(a))= t_{a}$ since $G_{F(a)}\subset R_f(a)$. Since $R_f(a)$ is $T-$invariant, one has $h_{top}^{P}(T,R_f^{T}(a)\cap U)=h_{top}^{P}(T,R_f(a)\cap U)=t_{a}$ by Lemma \ref{packing-lemma}. 
	
	For any $\mu$ with $\int f d\mu=a$, one has $G_{\mu}\subseteq R_f(a)$, so $h_{top}^{UC}(T,R_f^{T}(a)\cap U)=h_{top}^{UC}(T,R_f(a)\cap U)=h_{top}^{UC}(T,R_f^{T}(a))=h_{top}^{UC}(T,R_f(a))=h_{top}(T,X)$ by Theorem \ref{entropy}. 
	
	Finally, we claim that for any $a\in  L_f\setminus \{\int fd\mu_{\max}\},$ one has $t_{a}<h_{top}(T,X).$ Otherwise, there exists $\{\mu_n\}_{n\in \mathbb{N^{+}}}\subset\cM(X,T)$ such that $\int fd\mu_n=a$ for each $n\in\mathbb{N^{+}}$ and $\lim_{n\to \infty} h(T,\mu_n)=h_{top}(T,X).$ Since $\left\{\mu\in \mathcal M(X,T):\,  \int f d\mu=a\right\}$ is compact, we can assume that $\lim_{n\to \infty} \mu_n=\mu_0\in \cM(X,T)$ and $\int f d\mu_0=a.$ Then $h(T,\mu)=h_{top}(T,X)$ by the upper semi-continuity of the entropy map.
	Thus by the uniqueness of measure of maximal entropy, one has $\mu_0=\mu_{\max}$ which contradicts that $a\in  L_f\setminus \{\int fd\mu_{\max}\}.$ So we have $t_{a}<h_{top}(T,X),$
	This completes Theorem \ref{theo-level}.
\end{proof}

\nocite{*}

\medskip

$\mathbf{Acknowledgements}$. We thanks an anonymous referee for his/her careful reading and patience on the previous version of the manuscript. X. Hou and X. Tian are supported by National Natural Science Foundation of China Nos. 12071082, 11790273. Y. Zhang would like to thank Fudan University Key Laboratory
Senior Visiting Scholarship for support him visit Department of Mathematics, where part of the work was done.
Y. Zhang is supported by National Natural Science Foundation of China Nos. 11701200 and 11871262, and Hubei Key Laboratory of Engineering Modeling and Scientific Computing.


\end{document}